\documentclass[a4paper,10pt]{amsart}
\usepackage[utf8x]{inputenc}
\usepackage[english]{babel}
\usepackage[active]{srcltx}
\usepackage{verbatim}

\usepackage{amsmath,latexsym,amssymb,verbatim}

\usepackage{graphicx}

\newcounter{teoremaganso}

\newtheorem {bigtheo} [teoremaganso] {Theorem}

\newtheorem{lema}{Lemma}[section]
\newtheorem{theo}[lema]{Theorem}
\newtheorem{prop}[lema]{Proposition}

\newtheorem{defi}[lema]{Definition}
\newtheorem{rema}[lema]{Remark}

\newtheorem{exam}[lema]{Example}
\newtheorem{problem}[lema]{Problem}
\newtheorem{hyp}[lema]{Hypothesis}

\def\sideremark#1{\ifvmode\leavevmode\fi\vadjust{\vbox to0pt{\vss % the remark
      \hbox to 0pt{\hskip\hsize\hskip1em           %                will appear only
 \vbox{\hsize2cm\tiny\raggedright\pretolerance10000%                on the side
 \noindent #1\hfill}\hss}\vbox to8pt{\vfil}\vss}}}%
\newcommand{\R}{\mathbb{R}}
\newcommand{\C}{\mathbb{C}}
\newcommand{\Q}{\mathbb{Q}}

\newcommand{\Z}{\mathbb{Z}}

\title{Inductive solution of the \\
tangential center problem on zero-cycles}\author{A. \'Alvarez, J.L. Bravo, P. Marde\v si\'c}
\date{07/12/12}
\begin{document}

\begin{abstract}

Given a polynomial $f\in\C[z]$ of degree $m$, let $z_1(t),\ldots,z_m(t)$ 
denote all algebraic functions 
defined by $f(z_k(t))=t$. Given integers $n_1,\ldots,n_m$ such that $n_1+\ldots+n_m=0$, 
the tangential center problem on zero-cycles asks to
find all polynomials $g\in\C[z]$ such that $n_1g(z_1(t))+\ldots+n_mg(z_m(t))\equiv 0$. 
The classical Center-Focus Problem, or rather its
tangential version in important non-trivial planar systems lead to the above problem.

The tangential center problem on zero-cycles was  recently solved in a preprint by Gavrilov and Pakovich \cite{GP}. 

Here we give an alternative solution based on 
induction on the number of composition factors of $f$ under a generic hypothesis on $f$.
First we show the uniqueness of decompositions $f=f_1\circ\ldots\circ f_d$, such that every $f_k$
is $2$-transitive, monomial or a Chebyshev polynomial under the assumption that 
in the above composition there is no merging of critical values.  

Under this assumption, we give a complete (inductive) solution of the tangential center problem
on zero-cycles.
The inductive solution is obtained through three mechanisms: composition, primality and
vanishing of the Newton-Girard component on projected cycles. 
\end{abstract}

\subjclass[2010]{34C07, 34C08, 34M35, 14K20}

\keywords{Abelian integrals, tangential center problem, center-focus problem, moment problem}

\thanks{
The first author was partially supported by Junta de Extremadura
and FEDER funds. The second author was partially supported by
Junta de Extremadura and a MCYT/FEDER grant number MTM2008-05460.
The first two authors are grateful to the Universit\'{e} de
Bourgogne, for the hospitality and support during the visit when
this work was started. The third author thanks the Universidad de
Extremadura for the hospitality and support during the visit. }

\maketitle

\section{Introduction}\label{Section:1}
A classical problem in planar polynomial vector fields starting from Poincar\'e is the \emph{center-focus problem}. 
The problem asks for the determination of mechanisms leading to the creation of a center (a singularity surrounded 
by a continuous family of closed orbits) rather than  a focus (singularity attracting or repelling all nearby trajectories). 
The problem has not yet been solved in a satisfactory way except for quadratic vector fields \cite{D}, \cite{K}. 

The \emph{center-focus problem} has its \emph{infinitesimal version}:
Starting from a vector field $X_0$ having a center, find all perturbations $X_\lambda$ preserving the center.
More generally, given a vector field $X_0$ having a continuous family $\gamma_0(t)$ of closed orbits, 
determine the deformations $X_\lambda$ of $X_0$ preserving these closed orbits. That is, we ask  
the deformed  family $\gamma_\lambda(t)$  of orbits of $X_\lambda$ to be \emph{closed}, too. 

Taking a parametrized transversal $T$ to the family of closed curves $\gamma(t)$, one defines the \emph{displacement map} $\Delta$
 as the first return map minus identity along trajectories of $X_\lambda$. Then, a continuous family of closed curves is preserved, if the displacement map of the deformation along the chosen family of closed curves is identically zero.

The most popular family of polynomial systems having a continuous family of closed orbits is the family of Hamiltonian 
vector fields $X_F=-\frac{\partial F}{\partial y}\frac{\partial}{\partial x}+\frac{\partial F}{\partial x}\frac{\partial}{\partial y}$. 
The trajectories of $X_F$ lie in the level curves of the Hamiltonian. Consider their deformations $X_F+\epsilon Y$.
Taking a transversal $T$ to a family $\gamma(t)$ of closed curves of the Hamiltonian vector field $X_F$ parametrized by $F$, 
the displacement function of the above deformation is of the form

$$
\Delta(t,\epsilon)=-\epsilon\int_{\gamma(t)}\omega_Y+o(\epsilon),
$$
 where $\omega_Y$ is the dual form to the vector field $Y$ and $o(\epsilon)$
is a function depending on $t$, but tending to zero faster than $\epsilon$, for $\epsilon\to0$.
The function $t\mapsto I(t)=\int_{\gamma(t)}\omega_Y$, $\gamma(t)\subset F^{-1}(t)$, is an \emph{abelian integral}.

A necessary condition for having a solution of the infinitesimal center problem is the  
vanishing of its first-order term given by the abelian integral $I(t)$.

This motivates the \emph{tangential (or first-order) center problem}. We formulate it in its complex form.

\begin{problem} {\bf Tangential center problem.} Given a polynomial $F\in \C[x,y]$ and a continuous family $\gamma(t)$ of cycles belonging to the first homology group $H_1(F^{-1}(t)),$ find all polynomial forms $\omega$ such that the abelian integral 
\begin{equation}\label{ab2}
I(t)=\int_{\gamma(t)}\omega
\end{equation}
 vanishes identically.
\end{problem}

The problem was solved by Ilyashenko \cite{I} under a genericity assumption on $F$.
He proves that under a genericity condition on $F$, for any family of cycles $\gamma(t)\in H_1(F^{-1}(t))$,
 the abelian integral \eqref{ab2}  vanishes identically if and only if the form $\omega$ is relatively exact, i.e., of the form $\omega=GdF+dR$, for some $G,R\in\C[x,y]$.

Without the genericity hypothesis on $F$, the claim is false and the tangential center problem is open. 
One important non-generic case is the \emph{hyper-elliptic case} when $F(x,y)=y^2+f(x)$, $f\in\C[x]$. 
For hyper-elliptic $F$, the problem was solved by Christopher and the third author \cite{CM} under the
 hypothesis that the family of cycles $\gamma(t)$ is a family of vanishing cycles.
 The tangential center problem is open even in the hyper-elliptic case for general family $\gamma(t)$ of cycles.

Let us note a related problem studied by Gavrilov \cite{G} and Bonnet and Dimca \cite{BD}.

\begin{problem}{\bf Integrability problem.}  Given a polynomial $F\in \C[x,y]$, find all polynomial forms $\omega$ such that the abelian integral $I(t)=\int_{\gamma(t)}\omega$ vanishes identically along \emph{any} cycle $\gamma(t)\in H_1(F^{-1}(t))$.
\end{problem}
Note the difference: in the Integrability problem one asks for the vanishing of abelian integrals along \emph{all} cycles and not just one family of cycles as in the  Tangential center problem.
Under generic hypothesis on $F$, Ilyashenko \cite{I} proves that 
by monodromy one family of cycles generates all cycles. Hence 
the vanishing of abelian integrals along one family of cycles implies the vanishing along all cycles and the two problems coincide. 

In general it is false. It is easy to see that there are solutions of the tangential center problem due to the presence of a symmetry on a family of cycles. 
This solution will not necessarily be a solution for another family of cycles not respecting the symmetry.
The form is hence not relatively exact as
Abelian integrals of relatively exact forms vanish along any cycle of $\gamma(t)$.

This paper is dedicated to the study of the tangential center problem on zero-cycles:

\begin{problem}\label{tang}{\bf Tangential center problem on zero-cycles.} Given a polynomial 
$f\in\C[x]$ and a family of zero-cycles $C(t)$ of $f$, determine all functions $g\in\C[x]$ such that 
$\int_{C(t)}g$ vanishes identically. 
\end{problem}

Here, integration is just calculation of the value of a function at some points $z_i(t)$ determining the cycle $C(t)=\sum n_i z_i(t)$ (for more details see the next section).
Note that in this case abelian integrals on zero-cycles are in fact algebraic functions, so the problem is certainly easier 
than the initial problem. 

Nevertheless, in \cite{ABM} , we showed  that the tangential center problem for 
the hyper-elliptic case is directly related to the above
tangential center problem on zero-cycles for hyper-elliptic abelian integrals 
(see also \cite{GP}).

In \cite{ABM} we introduced the classes of \emph{balanced and unbalanced cycles}. 
We solved the  tangential center problem on zero-cycles by induction under the hypothesis that the 
initial cycle is unbalanced and that in the inductive process one encounters only unbalanced cycles. 
We called such a cycle \emph{totally unbalanced}.
The main result from \cite{ABM} can be resumed by saying that under the hypothesis that the 
cycle $C(t)$ is totally unbalanced, 
the only mechanism producing tangential centers is a composition (i.e. symmetry) 
mechanism or a sum of composition mechanisms. 
Our proof was based on results of Pakovich-Muzychuk \cite{PM} relative to the composition 
conjecture in the moment problem for the Abel equation. 
In particular we used Pakovich-Muzychuk's characterization of invariant irreducible spaces.
In this paper we deal with the remaining case of balanced cycles or cycles leading to balanced 
cycles in the induction process. 

While we were preparing the present paper, the preprint \cite{GP} of Gavrilov and Pakovich 
appeared, solving the general  tangential center problem on zero-cycles. 

Their solution is based on three steps:

\begin{description}
\item[ (i)] Description of all possible irreducible subspaces of $\Q^m$ invariant by the monodromy $G_f$ of $f$.
\item[(ii)] Given a cycle $\delta(t)=\sum_{i=1}^m n_i z_i(t)$ characterized by $(n_1,\ldots,n_m)\in \Z^m$, provide a method 
allowing to decompose the invariant space generated by the action of monodromy on $(n_1,\ldots,n_m)$ in a 
direct sum of irreducible $G_f$-invariant subspaces.
\item[(iii)]
For a given irreducible $G_f$-invariant subspace $V$, describe the space $Z_V$ consisting of polynomials $g$ such that $\int_{\delta(t)}g\equiv0$, for all $\delta(t)\in V$.
\end{description}

Our approach is different.
We solve  the tangential center problem on zero-cycles by induction on the number of composition factors of $f$.  Theorem A together with Theorem 2.3 of \cite{ABM} gives the basis of induction and Theorems B and C give the induction step. The solution is given  for arbitrary cycles, but under a generic hypothesis on the polynomial $f$. 
We think that our inductive approach sheds new light to the complicated structure of the space of solutions of the tangential center problem on zero-cyles. In particular it isolates three mechanisms  leading to a tangential centers for zero-cycles: 
composition, primality and
vanishing of the Newton-Girard component on projected cycles. We illustrate the complex structure of the solution of the tangential center problem by some examples. 

First let us explain the hypothesis under which our study is done.
It follows from the Burnside-Schur theorem that primitive (i.e., undecomposable) composition factors of a polynomial $f$ are of one of the following three types: $2$-transitive, linearly equivalent to a monomial $z^k$ or a Chebyshev polynomial $T_k$, with $k$ prime.
Note that a composition of monomials is a monomial and similarly a composition of Chebyshev polynomials is a Chebyshev polynomial. 
%The same holds if functions are linearly equivalent to Chebyshev functions  (by the same scaling). 

Given a polynomial $f$ let $f=f_0\circ\cdots\circ f_d$ be a decomposition of the polynomial $f$ in its composition factors, which are $2$-transitive, Chebyshev or monomial \emph{not necessarily of prime degree}.
 
We completely solve the tangential center problem on zero-cycles when for every $k=1,\ldots,d$ 
the critical values of $f_0\circ\cdots\circ f_{k-1}$ and $f_{k}$ do not 
merge in the composition $f_0\circ\cdots\circ f_k$ (see Definition \ref{dontmerge}).
\begin{comment}
the 
decomposition $f=f_1\circ\cdots\cdots f_k$ verifies the
following hypotheses:
\begin{enumerate}
 \item The images by $f_0\circ\ldots\circ f_{k}$ of the zeros of $(f_0\circ\ldots\circ f_{k-1})'(f_k(z))$
 and of the zeros of $f_k'(z)$ are disjoint. 
 \item $f_0\circ\ldots\circ f_{k-1}$ is injective on the set of critical values of $f_k$.
\end{enumerate}
\end{comment}
We prove that under the above hypothesis, the decomposition of $f$ as $f=f_0\circ\cdots\circ f_d$
is unique. 
The  hypothesis of non-merging of critical values allows us to decompose the monodromy group of $f_0\circ\ldots\circ f_k$
in a semidirect product of the monodromy group of $f_0\circ\ldots\circ f_{k-1}$
and the monodromy group of $f_k$. Under these hypotheses, we solve recursively the  tangential center problem on zero-cyles.
Using new methods, we show that in addition to the composition mechanisms, two new mechanisms leading to tangential centers exist. They both appear for balanced cycles on the level of the basis of recursion and in the recursion step.

One mechanism is related to the primality of  monomial or Chebyshev polynomial factors  with the characteristic polynomial $P_C(z)=\sum_{i=1}^m n_i z^{i-1}$ associated to the cycle $C(t)=\sum_{i=1}^m n_i z_i(t)$.  It is described by (1) of Theorem  \ref{theo:inductionsptep}  (basis) and Theorem \ref{theo:main2} (recursive step).
The other mechanism is related to the 
vanishing of the Newton-Girard component on projected cycles
in the case of $2$-transitive factors. It is described by (2) of Theorem  \ref{theo:inductionsptep} (basis)  and Theorem \ref{theo:main} (recursive step).

The motivation of the tangential center problem for the Abel equation
\begin{equation}\label{eq:abel}
x'=p(t)x^2+\lambda q(t) x^3, 
\end{equation}
proposed by Briskin, Françoise and Yomdin in a series of papers~\cite{BFY99}-\cite{BFY00-2},
was to obtain general mechanisms for the existence of centers. For $p,q$ 
polynomials, the tangential center problem has been totally solved
by Pakovich and Muzychuk~\cite{PM} proving that there is a tangential center
if and only if $\int_0^t q(s)\,ds$ can be written as a sum of polynomials having
a common factor with $\int_0^tp(s)\,ds$. Moreover, Briskin, Roytvarf and Yomdin~\cite{BRY} 
have proved that composition (just one summand) generates all centers of \eqref{eq:abel}
at infinity except for a finite number of exceptional cases. Nevertheless, there are centers and
tangential centers for which composition is not the generating
mechanism, like \eqref{eq:abel} when $p,q$ are not polynomials (see \cite{ABC,CGM}) or, in a more general
context,  (hyperelliptic) planar systems (see Example 9.2 of \cite{ABM}).
We hope the mechanims obtained here
will shed some light on these problems.

\section{Main results}\label{sec:main}

Let $f\in\C[z]$ be a polynomial of degree $m$.
Points $z \in \C$ such that $f'(z)=0$ are critical points of $f$. 
Corresponding values $t=f(z)$ are critical values and non-critical values are called regular values. 
Let $\Sigma$ denote the set of critical values of $f$. 
Then $f: f^{-1} (\C\setminus \Sigma)\to \C\setminus \Sigma$ is a fibration with zero-dimensional fiber. 
For any regular value of $t$ the fiber $f^{-1}(t)$ consists of $m$ distinct points. 
These points can be continuously transported along any path in $\C\setminus\Sigma$. 
Replacing in the fibration $f: f^{-1}(\C\setminus \Sigma)\to \C\setminus \Sigma$ the fibers $f^{-1}(t)$ 
by their 0-th homology groups $H_0(f^{-1}(t))$ or their reduced 0-th homology 
groups $\tilde H_0(f^{-1}(t))$, one obtains the homology or reduced homology fibration.

Let $t_0$ be a regular value of $f$ and $\pi_1 (\C\setminus\Sigma,t_0)$ the first homotopy 
group of the base with base point $t_0$. Then the continuous transport gives a morphism 
$\pi_1 (\C\setminus\Sigma,t_0)\to Aut(H_0(f^{-1}(t_0))$ from the first homotopy group of the 
base to the group of automorphisms of the homology fiber. Its image is called the monodromy 
group $G_f$ of $f$. This group is isomorphic to the Galois group of $f(z)-t$ seen as a polynomial
 over $\C(t)$ (see \cite[Th. 3.3]{CM}, for example).  Any section $C$ of the homology fibration is 
called a zero-chain. A section of the reduced homology fibration is called a zero-cycle.

Choosing an order $\{z_1(t),\ldots, z_m(t)\}$ among the $m$-preimages of $t$ by $f$, 
a zero-chain of $f$ is represented by 
\begin{equation}\label{C}
C(t) = \sum_{i=1}^m n_i z_i(t), \quad n_i \in \Z
\end{equation} 
and zero-cycles moreover satisfy
\begin{equation}\label{cycle}
\sum_{i=1}^m n_i =0.
\end{equation}
Note that $z_i(t)$, as well as $C(t)$ are multivalued functions on $\C\setminus\Sigma$. 
We also consider chains with coefficients in $\C$ when necessary. Then we will specify the coefficient ring. 

Let $C$ be a chain of $f$ and  $g$ a polynomial. The  Abelian integral of $g$ along a zero-cycle $C$ is defined by
\[
\int_{C(t)}g:=\sum_{i=1}^m{n_i}g(z_i(t)).
\]

This paper is dedicated to the solution of Problem \ref{tang} above.

In \cite{ABM} we introduced the notions of balanced and unbalanced cycles. We proved that if the cycle is unbalanced, then the tangential center problem is equivalent to solving some induced tangential center problems with $f$ replaced by some composition factors of $f$, hence polynomials of smaller complexity and  smaller degree. Let us precise the notions. In order to be able to perform the induction we need to extend the notions from cycles to chains. 
Let $\Gamma_m(f)\subset G_f$ denote the conjugacy class of $\tau_\infty$, where $\tau_\infty\in G_f$ 
is associated to a path winding once counter-clockwise around  infinity. 
That is, $\Gamma_m(f)$ is the set of all $\sigma \circ \tau_\infty \circ \sigma^{-1}$, 
for any $\sigma\in G_f$. We label the roots so that $\tau_{\infty}$, which is a permutation 
cycle of order $m$, shifts the indices of the roots by one, i.e., $\tau_{\infty}=(1,2,...,m)$.

\begin{defi}\label{defi:balanced}
We say that a chain $C(t)$ of $f$ is \emph{balanced} if
\[
\sum_{i=1}^m n_{p_i}\epsilon_m^i=0,\quad \text{for every }\tau=(p_1,p_2,\ldots,p_m)\in \Gamma_m(f),\quad 
\]
where $\epsilon_m$ is any primitive $m$-th root of unity. If $C(t)$ is not balanced, we say that 
$C(t)$ is \emph{unbalanced}.
 \end{defi}

The notion of balanced chain is independent on the way how permutation cycles are written. 
Let $\tau \in \Gamma_m(f)$ be of the form $\tau = \sigma \circ \tau_{\infty} \circ \sigma^{-1}$, $\sigma\in G_f$. Put $(p_1,p_2, \ldots, p_m)= (\sigma(1), \sigma(2), \ldots, \sigma(m))$ and 
\[
(n_{p_1}, n_{p_2}, \ldots, n_{p_m}) = (n_{\sigma(1)}, n_{\sigma(2)},\ldots, n_{\sigma(m)}) = \sigma(n_1, n_2, \ldots, n_m).
\]
Then, we can rephrase the definition of a chain $C(t)$ being \emph{balanced} as
\[
\sum_{i=1}^m n_{\sigma(i)}\epsilon_m^i=0,\quad \text{for every } \sigma \in G_f.
\]

\begin{defi}
We call the polynomial
\begin{equation}\label{char}
P_C(z)=\sum_{n=1}^m n_j z^{j-1}
\end{equation}
 \emph{characteristic polynomial} of the chain $C(t)$.
 \end{defi}

Note that if $C(t)$ is balanced, then $P_C(\epsilon_m)=0$. The converse does not hold in general.
The fact that the characteristic polynomial has the root $\epsilon_m$ 
depends on the election of the cycle of infinity, so 
indeed we should define the characteristic ideal as in \cite{ABM}, but in the special cases
we use the characteristic polynomial, $C(t)$ is balanced if and only if $P_C(\epsilon_m)=0$.

Assume that  $f=\tilde f\circ h$, for  $\tilde f,h\in \mathbb{C}[z]$, $\deg(h)=d$. 
Consider the imprimitivity system $\mathcal B=\{B_k\}_{k=1,\ldots,m/d}$, 
where $B_k=\{i \in \{1,2,\ldots,m\} |h(z_i(t))=h(z_k(t))\}$, associated to $h$.
(The general definition of imprimitivity system is recalled below in this section.)

\begin{defi}
The cycle ${h}({C}(t))$ of $\tilde{f}$ called the \emph{projection} of $C(t)$ by $h$ is defined by
\begin{equation}\label{def:cicloproyectado}
{h}(C(t)) = \sum_{{h}(z_i(t)) } \left(\sum_{{h}(z_j)={h}(z_i)} n_j\right){h}(z_{i}(t)) = \sum_{k=1}^{m/d} \left( \sum_{i\in B_k} n_i \right) w_k(t).
\end{equation}
Here $w_1(t),\ldots w_{m/d}(t)$ are all the different roots ${h}(z_i(t))$ of $\tilde f(z) = t$.
\end{defi}
With the  above notation, the main result of \cite{ABM}, Theorem~2.2 (ii), can be stated as follows.

\begin{theo}[\cite{ABM}]\label{theo:abm}
Let $f\in\mathbb{C}[z]$, and let $C(t)$ be an unbalanced cycle. Then, \[
\int_{C(t)}g\equiv 0, \quad {\text for } \, g\in\mathbb{C}[z]
\]
if and only if there exist $f_i,g_i,h_i\in \mathbb{C}[z]$, $1\leq i\leq d$ such that $\deg(h_i)>1$, $f=f_i\circ h_i$, $1\leq i\leq d$, $g=\sum_{i=1}^d g_i\circ h_i$ and $\int_{h_i(C(t))} g_i\equiv 0$. 
\end{theo}

\begin{rema}
If the cycle $C(t)$ is unbalanced, then the tangential center problem reduces to solving it for the projected cycles. If some of these projected cycles are unbalanced, Theorem~\ref{theo:abm} applies again. In particular, if all projected cycles (in every succesive step) are unbalanced, we call the initial cycle \emph{totally unbalanced} and the problem is completely solved by induction using Theorem~\ref{theo:abm}.
\end{rema}

In this paper we study the remaining case, when $C(t)$ is balanced or in the progress of projecting we arrive at some balanced cycle. It appears that the 
decompositions of the polynomial $f$ play a central role in the problem.

The action of the monodromy group $G_f$ of $f$ on the set of solutions $z_i(t)$ of the equation $f(z)=t$ is 
closely related to the decomposability of $f$.
When a group $G$ acts on a finite set $X$, it is said that the action is {\it imprimitive} if $X$ can be non-trivially decomposed into subsets 
of the same cardinality $B_i$ such that every $\sigma \in G$ sends each $B_i$ into $B_j$ for some $j$. Otherwise, the action is 
called {\it primitive}. The action of $G$ is said to be {\it $2$-transitive} if given any two pairs of elements of $X$, $(i,j)$ and $(k,l)$, 
there is an element $\sigma \in G$ such that $\sigma(i)=k$ and $\sigma(j)=l$. It is easy to prove that in these cases $X$ cannot be divided 
into disjoint subsets such that $G$ acts on them, which is one of the trivial cases of a primitive action. These disjoint sets $B_i$ are 
called \emph{blocks} and they form a partition of $X$ called an \emph{imprimitivity system}. This definition is consistent with the previous 
one (see e.g. \cite[Prop. 4.1]{ABM}).

Precisely, the action of $G_f$ on $\{z_i(t)\}$ is imprimitive if and only if $f$ is decomposable (see \cite[Prop. 3.6]{CM}). 
Burnside-Schur Theorem (see e.g. \cite[Th. 3.8]{CM}) classifies primitive polynomials (i.e., a polynomial that cannot be written as a composition of two polynomials of degree greater than one):

\medskip
\noindent {\bf Burnside-Schur Theorem} {\it Let $f$ be a primitive polynomial and $G_f$ its monodromy group. Then, one of the following holds:
\begin{enumerate}
\item The action of $G_f$ on the $z_i(t)$ is $2$-transitive. We call such a polynomial \emph{$2$-transitive}.
\item $f$ is (linearly) equivalent to a Chebyshev polynomial $T_p$ where $p$ is prime.
\item $f$ is (linearly) equivalent to $z^p$ where $p$ is prime.  
\end{enumerate}
}
\medskip

\begin{comment}
The decomposition of a polynomial in primitive factors is not unique. The lack of uniqueness is described  
by Ritt's theorem~\cite{R}. Indeed, using~\cite{PRitt} and \cite{MZ}, we show in Section~\ref{sec:decompositions} 
that a decomposition $f=f_1\circ f_2\circ\ldots\circ f_r$ is unique when each factor $f_i$ is, up to pre- and post- 
composition by linear functions, of one of the following families:
\begin{enumerate}
 \item $2$-transitive
 \item Monomial
 \item Chebyshev
 \item $z^{sm}R^m(z^m)$, for $s,m\in\mathbb{N}$ and $R\in\mathbb{C}[z]$
 \item Caso raro y largo
\end{enumerate}
\end{comment}

We give a solution of the tangential center problem by induction on the number of composition factors $f_i$ under some 
additional hypotheses. First we solve the problem for basic composition factors: 
$2$-transitive, monomial or Chebyshev of \emph{not necessarily prime degree} (basis of induction). Monomial and Chebyshev cases are similar. They are hence treated together. 

We introduce first some notations. Given a polynomial $f$ of degree $m$ and $z_i(t)$, $i=1,\ldots,m$ the roots of $f(z)=t$, let $s_k$, $k=0,\ldots,m-1$ denote the sums of $k$-th powers of roots of $f-t$:
\[
s_k=\sum_{i=1}^{m} z_i^k(t).
\]
Note that the Newton-Girard formulae express $s_0,\ldots,s_{m-1}$ in function of the coefficients of $f$. In particular they are independent of $t$. Let $s(f)=(s_0,\ldots,s_{m-1})\in\mathbb{C}^m$ denote the Newton-Girard vector of $f$. 

When convenient, using the division algorithm, we will write $g(z)$ as $g(z)=\sum_{k=0}^{m-1}g_k(f)z^k$. That is, we express $g(z)$ as 
a polynomial in $z$ of degree $\deg g<m=\deg f$, with coefficients in $\mathbb{C}[f]$.
Substituting $f=t$ for integration each coefficient $g_k$ becomes a polynomial $g_k(t)$ in $t$.

We denote $g(t)=(g_0(t),\ldots,g_{m-1}(t))$. Let $< - , - >$ denote the scalar product in 
$\C^m$: $<g(t),s(f)>=\sum_{i=0}^{m-1}g_i(t) s_i$.

\begin{bigtheo}\label{theo:inductionsptep}
Assume that $C(t)=\sum n_jz_j(t)$ is a balanced chain of a polynomial $f$. 
\begin{enumerate}
\item If the action of $G_f$ is $2$-transitive, 
then there exists $n$ such that $n_j=n$ for every $1\leq j\leq m$.
Moreover, for every 
$g(z)=\sum_{k=0}^{m-1} g_k(f) z^k$, $\int_{C(t)}g\equiv 0$ if and only if
$
g(t)\in s(f)^\perp.
$
\item  If $f(z)=z^m$ (resp. $f(z)=T_m(z)$), then in the solution $g=\sum_{j=0}^{m-1} 
g_j h_j(z)$, with $h_j(z)=z^j$, (resp. $h_j(z)=T_j(z)$) of the tangential center problem $\int_{C(t)}g\equiv 0$ the only terms $h_j$ present 
are for $j$ such that $k=g.c.d.(j,m)$ such that $P_C(\epsilon_m^k)=0$, where $\epsilon_m$ is a primitive $m$-th root of unity.
\end{enumerate}

\end{bigtheo}

Condition (2) can also be written as:
If $f(z)=z^m$ (resp. $f(z)=T_m(z)$), then $\int_{C(t)}g\equiv 0$ if and only if $g(z)=\sum_{j=0}^{m-1} g_j h_j(z)$, where 
$g_j \in \mathbb{C}[f]$, $h_j(z) = z^j$ (resp. $h_j(z)=T_j(z)$), for any $j$ is such that $\Phi_{m/k}(z) | P_C(z)$, where 
$k=g.c.d.(m,j)$,  $P_C(z)=\sum_{i=1}^m n_i z^{i-1}$ is the characteristic polynomial of the chain $C$ and $\Phi_n(z)$ is the 
$n$-th cyclotomic polynomial, that is,
\[
\Phi_n(z) =  \underset{\substack{\quad 1\leq j\leq n,\\gcd(j,n)=1}}{\prod} ( z -
e^{\frac{2 \pi i}{n}j} ).
\]

\begin{rema}\label{ideasA} 
(1) We explain here the geometric idea behind the proof of Theorem \ref{theo:inductionsptep}: Let $C(t)$ be a balanced chain.
In order to prove (1), one considers the case when the monodromy group $G_f$ of $f$ is $2$-transitive and one fixes any root $z_i(t)$ and 
considers its stabilizer $H_i$. The assumption that $G_f$ is $2$-transitive means that the stabilizer $H_i$ of the root acts transitively on 
the other roots. Averaging by the action of the stabilizer $H_i$, the condition that the chain $C(t)$ is balanced, gives a relation independent 
of $i$. From this relation it follows that in the chain $C(t)=\sum_{i}n_i z_i(t)$, all $n_i$ coincide. If the chain is a cycle, the cycle is 
trivial. 

Once we know that all $n_i$ in a chain coincide, the characterization of the vanishing of the integral 
$\int_{C(t)}g$ by the orthogonality of the vector $g(t)$ to the Newton-Girard vector $s(f)$ is just rewriting the vanishing of the integral. 

(2) If $f(z)=z^m$, then the roots $z_i$ are simply given by the roots of unity $\epsilon_m^i$. By explicit calculation $\int_{C(t)}g = 
\sum_{j=0}^{m-1}g_j s^jP_C(\epsilon_m)$ and the theorem follows. The orbit of $C(t)$ by monodromy is a sum of irreducible invariant spaces. 
In this case the irreducible invariant spaces are easily calculated and are related to the divisors of $m$. The result follows from explicit 
calculations. 

If $f(z)=T_m(z)$, then $\int_{C(t)}g=\alpha_j(t)P_C(\epsilon_m^j)+\overline{\alpha_j(t)}P_C(\overline{\epsilon_m^j})$, with $\alpha_j(t) = 
\frac12 e^{i\xi(t)\frac{j}{m}}$ and $\xi(t)=\arccos(t)$. This case is next treated similarly as the monomial case.  
\end{rema}

\medskip

Next we show how to reduce the tangential center problem on a cycle $C(t)$ of $f=\tilde f\circ h$, where $h$ is $2$-transitive or linearly 
equivalent to a monomial or a Chebyshev polynomial of not necessarily prime degree. 
In the step of induction, we reduce the original tangential center problem of a cycle of $f$ to the tangential center problem of the projected 
cycle $h(C(t))$, which is a cycle of the function $\tilde f$, simpler than the original function $f$. We do the induction step under the 
generic hypothesis that the critical values of $\tilde f$ and $h$ do not merge:

\begin{defi} \label{dontmerge} 
Let $\tilde f$ and $h$ be two nonlinear polynomials and let $f=\tilde f\circ h$. We say that 
the \emph{critical values of $\tilde f$ and $h$ do not merge} if
\begin{enumerate}
\item 
$\{f(z)\colon \tilde f'(h(z))=0\}\cap \{f(z)\colon h'(z)=0\}=\emptyset$,
\item
$\tilde f$ is injective on the set of critical values of $h$.
\end{enumerate}
\end{defi}

As we prove in Lemma~\ref{lema:semidirect}, condition (1) of Definition \ref{dontmerge} assures that the monodromy group $G_f$ is a semidirect 
product of the subgroups $N_h$ and $G_{\tilde f}$, $G_f=N_h\rtimes G_{\tilde f}$, where $N_h$ is the normal closure of the monodromy group 
$G_h\subset G_f$ with the natural injection (which in general is not a group morphism). Condition (2) of Definition \ref{dontmerge} assures 
that there exists a group morphisms $\phi:G_h\to G_f$.

We divide the induction step in two cases, according to the nature of $h$. We study first the case when  $h$ is $2$-transitive. 

\begin{bigtheo}\label{theo:main}
Let $f(z)$ be a polynomial such that $f=\tilde{f}\circ h$ with $\tilde f,h\in\mathbb{C}[z]$, $1<\deg(h)=d<m=\deg f$. Let $C(t)\in \tilde H_0(f^{-1}(t))$ be a balanced cycle of $f$.
 Assume that the critical values of $\tilde f$ and $h$ do not merge and that $h$ is $2$-transitive.

 Then $g(z)=\sum_{i=0}^{d-1} z^i g_i(h(z))$, verifies
\begin{equation}\label{g}
\int_{C(t)} g\equiv 0 \quad
\end{equation}
if and only if 
\begin{equation}\label{tildegg}
\int_{h(C(t))}\tilde g\equiv 0,
\end{equation}
and $g_i\in\mathbb{C}[w]$, $i=0,\ldots,d-1$, are solutions of the linear system 
\begin{equation}
\label{system}
<g,s>=\tilde g(w).
\end{equation}
\end{bigtheo}

\begin{rema}\label{NG} Thus, the solution $g$ and the inductive solution $\tilde g(w)$ are related by the linear system of equations $\tilde g(w)=\sum_{i=0}^{d-1} 
s_i g_i(w)$. Here $s_i$, $i=0,\ldots,d-1$ are expressed through the coefficients of $h$ by the Newton-Girard formulae for $h(z)-w$ and they are independent of $w$. 
We call $s(h)=(s_0,\ldots,s_{d-1})\in\C^{d}$ the Newton-Girard vector of $h$. Let $\pi_h(g)=\frac{<g(w),s(h)>}{|s(h)|}$ be the component of 
the vector  $g(w)=(g_0(w),\ldots,g_{d-1}(w))\in\C[w]^{d}$ representing $g$ in the direction of the Newton-Girard vector $s(h)$. Theorem 
\ref{theo:main} can be resumed by saying that under above conditions, \eqref{g} is equivalent to the vanishing of $\int_{h(C(t))}\pi_h(g)$,
where $\pi_h(g)$ is the component of the vector representing $g$ in the direction of the Newton-Girard vector of $h$. See Example \ref{examB}. 
\end{rema}

\begin{rema}\label{remB} 
The geometric idea of the proof of Theorem \ref{theo:main} is similar to the one of Theorem \ref{theo:inductionsptep} in the case of $h$ $2$-transitive. One 
applies averaging by the stabilizer $H_{i_0,j_0}$ of a root $z_{i_0,j_0}$ to the identity satisfied by balanced cycles. The roots 
are regrouped in groups defined by the imprimitivity system of $h$. One does not prove now that all coefficients $n_{ij}$ of the cycle are the 
same, but that all coefficients $n_{i_0,j}$, corresponding to roots in the same block $B_{i_0}$, are the same 
(Proposition~\ref{prop:balanced}). Next, given $g(z) = \sum_{i=0}^{d-1}z^ig_i(h(z))$, in the integral $\int_{C(t)}g$ one regroups 
all terms according to the block $B_{i_0}$ to which they belong. This allows to express $\int_{C(t)}g$ through an integral on the projected 
cycle $h(C(t))$ and the condition $\int_{C(t)}g=0$ becomes equivalent to the vanishing of $\int_{h(C(t))}\pi_h(g)$.
\end{rema}

Assume that  $f=\tilde f\circ h$, for  $\tilde f,h \in \mathbb{C}[z]$, $\deg(h)=d$. In order to formulate the induction step in the case of monomial or Chebyshev $h$ (Theorem \ref{theo:main2}), in addition to the projection of a cycle 
$C(t)$ by $h$ we need the notion of $h$-invariant parts of $C(t)$.

\begin{defi}\label{part}
For each cycle $C(t)$ of $f$ and the decomposition $f=\tilde f\circ h$ of $f$, consider the imprimitivity system  
$\mathcal{B} =\{B_1,\ldots,B_{m/d}\}$ associated to $h$.  We define the \emph {$h$-invariant parts of $C(t)$} as the 
chains $C_k(t) = \sum_{i\in B_k} n_i z_i(t)$ of $f$, or as chains  $\tilde C_k(w)=\sum_{i\in B_k}n_i z_i(w)$, $h(z_i)=w$, 
of $h$.
\end{defi}

The $h$-invariant part $C_k(t)$ corresponds to the part of the cycle presented by the roots in the $k$-th line in Figure~\ref{mono}.
Relations $C(t)=\sum_{k=1}^{m/d}C_k(t)$ and $\tilde C_k(w)=C_k(\tilde f(w))$ hold. Note that even if $C(t)$ is a cycle, 
the $h$-invariant parts $C_k(t)$ are only chains in general.

\begin{bigtheo}\label{theo:main2}Let $f(z)$ be a polynomial such that $f=\tilde{f}\circ h$ with $\tilde f,h\in\mathbb{C}[z]$, $1<\deg(h)=d<m=\deg f$. Let $C(t)\in \tilde H_0(f^{-1}(t))$ be a balanced cycle of $f$.
Assume that the critical values of $\tilde f$ and $h$ do not merge and  let $h$ be a monomial or a Chebyshev polynomial. 
Let $\tilde z_k(w)$, $k=1,\ldots,d$, denote the zeros of $h(z)-w$. 
Then the $h$-invariant parts $\tilde C_{k}(w)$, $k=1,\ldots,m/d$, of $C(t)$ are balanced.  

Moreover,
\[
\int_{C(t)} g \equiv 0
\]
if and only if $g$ is of  the form 
\[
g(z)=\tilde g(h(z))/(d-1)+u(z),
\]
where $\tilde g(w)$ is a polynomial such that
\begin{equation}\label{tildeg}
\int_{h(C(t))}\tilde g\equiv 0,
\end{equation}
and
\begin{equation}\label{gg}
\int_{\tilde C_{k}(w)} u\equiv 0,\quad 1\leq k\leq m/d.
\end{equation}
\end{bigtheo}

\begin{rema}\label{remaC}
The geometric idea of the Proof of Theorem \ref{theo:main2} is the following. One considers the imprimitivity decomposition $\{B_1,\ldots, B_m\}$ of the roots  associated to $h$. Fix $1\leq i_0\leq 
m$. Let $H_{i_0,j}$, $j=1,\ldots,d$, be the stabilizers of the root $z_{i_0,j}\in B_{i_0}$ and $H_{i_0}=\cap_{j=1,\ldots,d} 
H_{i_0,j}$ their intersection. Averaging with respect to $H_{i_0}$ the identity satisfied by balanced cycles, one shows first 
that the $h$-invariant parts are balanced. Moreover, averaging with respect to $H_{i_0}$ in the zero-dimensional abelian 
integral, one obtains that there exists a polynomial $p_{i_0}$ such that the vanishing of the abelian integral is equivalent to 
the integral on the $h$-invariant parts being equal to $p_{i_0}$. For these systems, we prove that the general solution 
is equal to the sum of a particular solution $\tilde g(h(z))/(d-1)$ of the non-homogeneous system with the general solution 
$u$ of the homogeneous system.
\end{rema}

Note that as $\tilde C_k(w)$ are balanced chains of $h$, 
\eqref{gg} is solved in Theorem~\ref{theo:inductionsptep} and the solution of \eqref{tildeg} is given by the induction hypothesis.

We prove Theorems~\ref{theo:main} and \ref{theo:main2} in Section \ref{section:demo}.

If $\tilde f$ and $h$ merge, then Theorems~\ref{theo:main} and \ref{theo:main2} also 
gives tangential centers. However, there can be other solutions not covered by them. Indeed,  
by perturbing $f = \tilde f \circ h$ we can assure that $\tilde f$ and $h$ do not merge and 
Theorems~\ref{theo:main} and \ref{theo:main2} gives solutions of the deformed system, which in the limit give solutions of the original system.

It follows from the Burnside-Schur Theorem that any polynomial $f$ can be decomposed as 
$f=f_0\circ\cdots\circ f_d$, where each factor $f_k$ is either $2$-transitive, linearly equivalent to
a monomial or linearly equivalent to a Chebyshev polynomial. Note that a composition of monomials is a monomial and
a composition of Chebyshev polynomials is a Chebyshev polynomial. 
We do not assume that the degree of these polynomials $f_k$ linearly equivalent to a monomial or to a Chebyshev polynomial is prime.  
Putting together Theorem~\ref{theo:abm} and Theorems~\ref{theo:inductionsptep}, \ref{theo:main} and \ref{theo:main2}, we solve inductively 
the tangential center problem under the following hypothesis on $f$:

\begin{hyp}\label{hyp}
Let $f\in\C[z]$ have a 
 decomposition $f=f_0\circ\cdots\circ f_d$, with $f_k$ $2$-transitive, linearly equivalent to a 
monomial or to a Chebyshev polynomial.  We assume that the critical values of 
$f_0\circ\cdots\circ f_{k-1}$ and $f_k$ do not merge
for any $1\leq k\leq d$.
\end{hyp}

We show that under Hypothesis \ref{hyp}, there exists a unique decomposition 
$f=f_0\circ\cdots\circ f_d$, with $f_k$ $2$-transitive, linearly equivalent to a monomial or to a Chebyshev polynomial, satisfying Hypothesis~\ref{hyp}. More precisely: 

\begin{prop}\label{unicity}

Assume that $f=f_0\circ\ldots\circ f_d$ for some polynomials $f_0,\ldots,f_d$ such that every $f_k$ is either $2$-transitive, linearly 
equivalent to a monomial or linearly equivalent to a Chebyshev polynomial, and that Hypothesis~\ref{hyp} holds.
If  $f=\tilde f_0\circ\ldots\circ \tilde f_{\tilde d}$ is another decomposition of $f$ for some polynomials $\tilde f_0, \ldots,
\tilde f_{\tilde d}$ such that every $\tilde f_k$ is either $2$-transitive, linearly equivalent to a monomial or linearly equivalent to a Chebyshev polynomial, 
then there exist $j_{k,1},\ldots,j_{k,d(k)}$, $j_{k,1}=j_{k-1,d(k-1)}+1$ ($k>0$), such that (up to linear transformations)
\[f_k = \tilde f_{k,1} \circ \ldots 
\circ \tilde f_{k,d(k)},\quad \text{for every } k=0, 1, \ldots, d.\] 

Moreover, if the critical values of $\tilde f_0\circ \ldots \circ \tilde f_{k-1}$ and $\tilde f_k$
do not merge for any $1\leq k\leq \tilde d$, then $d=\tilde d$, and (up to linear transformations)
\[ 
f_k=\tilde f_k,\quad \text{for every }k=0, 1, \ldots,d.
\]
\end{prop}

A decomposition of $f$ in primitive polynomials is not unique in general. 
To get uniqueness, we have to regroup successive monomials or Chebyshev
factors (or factors linearly equivalent to one of them) that commute (see~\cite{PRitt}).
Ritt's Theorem states that if a polynomial $f$ admits two different decompositions, then
there exist monomials or Chebyshev factors (or a factor linearly equivalent 
to one of the previous types) in a decomposition of $f$ into primitive polynomials
that commute. Note for instance that $z^6=(z^2)^3=(z^3)^2$. Section~\ref{sec:mon},
where we prove Proposition~\ref{unicity}, deals with the converse problem, that is, given a 
decomposition of $f=\tilde f\circ h$, Hypothesis~\ref{hyp} ensures
that any other decomposition of $f$ is a further decomposition of $\tilde f$ and $h$,
or, in other words, there exists no commuting factors between $\tilde f$ and $h$.

To conclude this section, we give two examples of the application 
of Theorems~\ref{theo:inductionsptep}, \ref{theo:main} and \ref{theo:main2}.

\begin{exam}\label{examB}
First, let us consider a polynomial composition of two 2-transitive polynomials $f=\tilde f\circ h$,
with $\tilde f(z)=z^3 - z^2 + z$ and $h(z)= z^3 + 2 z^2 - 1$, and the 
cycle 
\[C(t)=z_1(t)-z_2(t)+z_4(t)-z_5(t)+z_7(t)-z_8(t),\] 
where we are 
assuming that $(1,2,3,4,5,6,7,8,9)$ is the permutation associated to
a loop around infinity. Then, the imprimitivity systems 
are the equivalence classes modulo divisors of $m$, and 
it can be easily checked that Hypothesis~\ref{hyp}
holds. Moreover, $C(t)$ is balanced. 
\begin{figure*}[ht]
\begin{center}
{\scalebox{.15}{\includegraphics{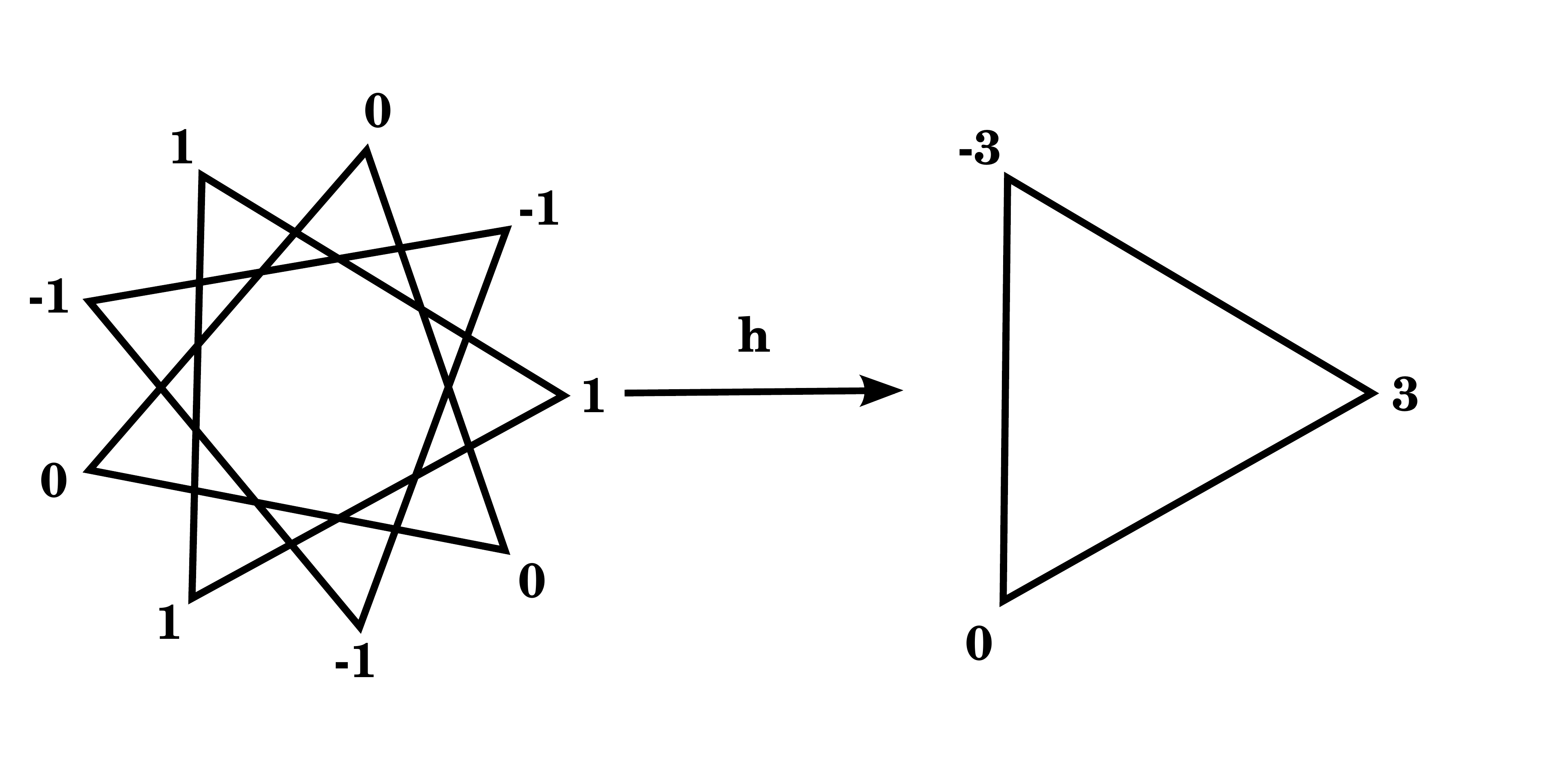}}}
\end{center}
\caption{Cycles $C(t)$ and $h(C(t))$ in Example~\ref{examB}}\label{cicle}
\end{figure*}

Write $g(z)=g_0(h(z))+zg_1(h(z))+z^2g_2(h(z))$, $g_i\in\mathbb{C}[w]$. By Theorem~\ref{theo:main},
$g$ satisfies $\int_{C(t)} g\equiv 0$ if and only if
\[
\int_{h(C(t))}\tilde g\equiv 0,
\]
where $h(C(t))=w_1(t)-w_2(t)$ is unbalanced, and $\tilde g(w)=3 g_0(w)-2g_1(w)+4g_2(w)$.

Now, the solutions of $\int_{h(C(t))}\tilde g\equiv 0$ are given by Theorem~\ref{theo:abm}.
As $\tilde f$ is primitive, the solutions are of the form
 $\tilde g(w)=k_0(\tilde f(w))$, 
for any polynomial $k_0\in\mathbb{C}[w]$.

Then $\int_{C(t)} g\equiv 0$ if and only if
$k_0(\tilde f(w))=3 g_0(w)-2g_1(w)+4g_2(w)$. For instance, if $k_0(w)=1$, and 
$g_0,g_1,g_2$ are quadratic polynomials: $g_i(z)=a_{i2}z^2+a_{i1}z+a_{i0}$,
then the space of solutions is $6$-dimensional given by
\[
\begin{split}
g(z) &=
\frac{1 + 2 a_{10} - 2 a_{11} + 2 a_{12} - 4 a_{20} + 4 a_{21} - 
   4 a_{22}}{3}+\left( a_{10} - a_{11} + a_{12}\right)z
\\&+\frac{1}{3} \left(4 a_{11} - 8 a_{12} + 3 a_{20} - 11 a_{21} + 
   19 a_{22}\right)z^2+\frac{4}{3} (2 a_{11} - 4 a_{12} - a_{21} + 2 a_{22})z^3
\\&+\left( a_{11} + 
 \frac{2}{3} (a_{12} + 3 a_{21} - 14 a_{22})\right)z^4+\frac{1}{3}\left(20 a_{12} + 3a_{21} - 
 22 a_{22}\right)z^5
\\&+\frac{2}{3} (7 a_{12} + 4 a_{22})z^6+\left( a_{12} + 4 a_{22}\right)z^7+ a_{22}z^8
\end{split}
\]
for any $a_{10},a_{11},a_{12},a_{20},a_{21},a_{22}\in\mathbb{C}$.
\end{exam}

\begin{exam}\label{examC}
Let us consider a polynomial composition of a 2-transitive polynomial and a monomial, $f=\tilde f\circ h$,
with $\tilde f(w)=w^3 - w^2 + w$ and $h(z)= z^6$, and the 
cycle 
\[C(t)=z_1(t)-z_2(t)+z_7(t)-z_{8}(t)+z_{13}(t)-z_{14}(t),\] 
where we are assuming that $(1,2,\ldots,17,18)$ is the permutation associated to
a loop around infinity. It can be checked that Hypothesis~\ref{hyp}
holds. Moreover, $C(t)$ is balanced. 

It is easy to see that $h(C(t))=3w_1(t)-3w_2(t)$ is an unbalanced cycle of $\tilde f$. Then, the 
solutions of 
\[
\int_{h(C(t))} \tilde g\equiv 0
\]
are given by Theorem~\ref{theo:abm}. More precisely, as $\tilde f$ is primitive, it follows that $\tilde g(w)$
is a function of $\tilde f$. Let $\tilde g(w)=g_0(\tilde f(w))$. On the other hand, 
the $h$-invariant parts of the cycle $C(t)$ are given by $\tilde C_1(t)=w_1(t)+w_3(t)+w_5(t)$, $\tilde C_2(t)=-w_1(t)-w_3(t)-w_5(t)$ and 
$\tilde C_3(t)=0$. Hence $P_{\tilde C_1}(w)=1+w^2+w^4$, $P_{\tilde C_2}=-P_{\tilde C_1}$ and $P_{\tilde C_3}=0$.

If $u=\sum_{i=1}^{5} u_j z^j$, where $u_j \in \mathbb{C}[z^6]$, by Theorem~\ref{theo:inductionsptep} (2), the 
solutions of
\[
\int_{\tilde C_k(t)} u\equiv 0\quad \text{for } k=0, 1, 2,
\]
are given by $u_0 = u_3 = 0$.

Finally, by Theorem~\ref{theo:main2}, the solutions of 
\[
\int_{C(t)}g\equiv 0
\]
are 
\[
g(z)=g_0(f(z))+u(z),
\]
for any polynomials $g_0$ and  $u=\sum_{i=1}^5 u_j z^j$ such that $u_0 = u_3 =0$.

\end{exam}

\section{Structure of the space of solutions}

In this section we give the general structure of the space of solutions. It will be used in the next section for solving the 
tangential center problem in two basic cases: the monomial and the Chebyshev case.

Choosing a basis $z_1(t),\ldots, z_m(t)$ in $H_0(f^{-1}(t))$, one can identify each chain $C(t)=\sum_{i=1}^m n_i z_i(t)$ with 
$n(C)=({n_1},\ldots, {n_m}) \in \mathbb{C}^m$. Similarly, to each vector $v=(v_1, \ldots,v_m) \in \C^m$ we associate the chain 
$C_v:= \sum_{i=1}^m v_i z_i(t)$. Then,
\[ 
\int_{C(t)}g=<(g(z_1(t)),\ldots,g(z_m(t))),\bar{n}(C(t))>,
\]
where $< - , - >$ is the usual scalar product in $\C^m$ and $\bar{n}$ denotes the complex conjugate of $n$. 

Given $f$ and $C$ we search for all $g$ such that $\int_{C(t)}g\equiv0$. Let $G_f$ be the monodromy group of $f$ and 
$V(t)\subset H_0(f^{-1}(t))$ the vector space generated by the orbit of a chain $C(t)$ of $f$ by $G_f$. By analytic continuation, the vanishing of $\int_{C(t)} 
g$ is equivalent to the vanishing of $\int_{\sigma(C(t))}g$, for any $\sigma(C(t))\in V(t)$. Let $H^0(f^{-1}(t))$ be the dual 
space to $H_0(f^{-1}(t))$. Hence, from abstract point of view, the tangential center problem is simply the problem of 
determining $(V(t)^\perp)^*$: the dual space to the orthogonal complement of $V$. The space $H^0(f^{-1}(t))$ is organized as 
an $m$-dimensional $\C[t]$-module, with multiplication defined by 
\[
P(t)g(z)=P(f(z))g(z).
\]

Let $V_r$ denote $r$-periodic (mod $m$) vectors in $\Q^m$ and let $D(f)$ be the set of positive divisors $d$ of $m=deg(f)$ 
such that there exists a decomposition $f= \tilde f \circ h$, where $deg(h)=d$. The structure of the $G_f$-invariant subspaces 
of $\Q^m$ is determined by Lemma 5.1 of \cite{PM}.

\begin{lema}[\cite{PM}]\label{Pako}
Any $G_f$-irreducible invariant subspace of $V(t)$ is of the form 

\[
U_r(t) = V_r(t) \cap ( V_{r_1}(t)^\perp\cap\cdots V_{r_l}(t)^\perp ),
\] 
where $r \in D(f)$ and $\{r_1, \ldots, r_l\}$ is a complete set of divisors of $m$ covered by $r$, that is, they are 
all the maximal divisors of $r$ in $D(f)$. The subspaces $U_r(t)$ are mutually orthogonal and any $G_f$-invariant subspace of $\Q^m$ 
is a direct sum of some $U_r(t)$ as above.
\end{lema}

For any natural $k$, $w_k$ denotes the complex vector $(1,\epsilon_m^k, \epsilon_m^{2k}, \ldots, \epsilon_m^{(m-1)k}) \in 
\C^m$, where $\epsilon_m = \exp(2 \pi i/m)$ is a primitive $m$-th root of unity. For $k=1,\ldots,m$ these vectors are 
orthogonal and form a basis of $\C^m$.

A choice of the chain $C(t)$ determines the corresponding invariant space $V(t)$. Note, moreover, that if a chain $C(t)$ has 
real coefficients $n_j$, then for any $k$ the vectors $w_k$ and $w_{m-k}$ simultaneously belong to $V(t)^{\C}$ or 
$(V(t)^\perp)^{\C}$.

Let us recall that $(V_r)^{\C}$ is generated by the vectors $w_k$ such that $m/r$ divides $k$. From now on, when necessary we 
will assume that all $\Q$-vector spaces ($V(t)$, $V(t)^{\perp}$, $V_d$, $U_d$, etc) are complexified.

Let $C_k:=\bar{w}_k$ and let $\{C_k: k\in S\subset\{1,\ldots,m\}\}$ be a basis of the space $V(t)$. Then the solution of 
the tangential center problem is given by the space generated by the dual basis $C_{k}^*$, $k\in\{1,\ldots,m\} \setminus S$.

Let $P_C$ denote the characteristic polynomial $P_C(z)=\sum_{j=1}^m n_j z^{j-1}$ and let 
\[
\Phi_j (z) = \prod_{1\leq k<j,\; gcd(k,j)=1}(z-e^{\frac{2\pi i}{j}k})
\] 
be the $j$-th cyclotomic polynomial. Note that $P_C(\epsilon_m^k) = <C,\bar{w}_k> = <w_k, C>$.

\begin{lema}\label{general}
Let  $f(z)$ be a polynomial of degree $m$ and $C(t)$ a chain of $f$. Then:
\begin{enumerate}
\item The $\mathbb{C}[t]$-module of solutions of the tangential center problem is given by
\[
U^*=U_{d_1}^*\oplus\cdots\oplus U_{d_j}^*.
\]
where a basis of $U_1^*$ is given by $C_m^*$ and $U_m^*$ contains functions $C_k^*$, with $k$ relatively prime 
with $m$. For any $d_j$ divisor of $m$, $U_{d_j}^*$ contains functions of the form $C_k^*$, with $k$ a multiple 
of $m/d_j$, but not of any prime factors of $d_j$, i.e., $g.c.d.(m, k)=m/d_j$.

\item The subspace $U_1^*$ is a subspace of $U^*$ if and only if $C(t)$ is a cycle. The subspace $U_m^*$ is a subspace of $U^*$ if and only if $C(t)$ is balanced. 

\item The above spaces $U_j^*$ are mutually orthogonal. Their dimensions satisfy 
\[
dim_{\C} (U_1^*)=1, \quad dim_{\C} (U_m^*)\geq \phi(m), \quad dim_{\C} (U_{d_l}^*)\geq \phi(d_l).
\] 
\end{enumerate}
\end{lema}

\begin{proof}
Assume first that $C(t)$ is a cycle. By definition of a cycle, this means that the chain $C_{m}$ associated to 
$w_m=(1,\ldots,1)$ belongs to $V(t)^\perp$. That is, $P_C(\epsilon_m^m) = P_C(1) = 0$. This is equivalent to $\Phi_1(z)$ 
being a factor of $P_C(z)$ and $C_m^*$ belonging to the space of solutions.

A chain $C(t)$ is unbalanced if for any numeration of the roots corresponding to a permutation cycle at infinity one has 
$\sum_j n_j \epsilon^j \ne 0$, i.e., the chain ${C}_{1}$ does not belong to $V(t)^\perp$.
Assume now that $C(t)$ is balanced. This means that the chain ${C}_{1}$ also belongs to $V(t)^\perp$, that is, 
$P_C(\epsilon_m^1)=0$. By Lemma \ref{Pako}, $U_m= V_{r_1}^\perp\cap\cdots\cap V_{r_l}^\perp$, where $r_i$ are all maximal 
divisors of $m$ in $D(f)$, but strictly smaller than $m$. This means that $m/r_i$ are prime factors of $m$ in $D(f)$. Then $U_m$ contains $C_{k}$, where $k$ is not a multiple of any prime factor of $m$ in 
$D(f)$. In particular, $U_m$ contains $C_k$ when $k$ is coprime with $m$. Note then that the vector $w_1$ belongs to $U_m$, and 
therefore $U_m \subset V(t)^\perp$. That gives that the functions  $C_k^*$ for $k$ coprime with $m$ are in the space of 
solutions of the center problem and $\Phi_m(z)$ divides $P_C(z)$.

Finally, we calculate the remaining elements of the basis of the solution vector space. They correspond to some irreducible 
components of $(V^{\perp})^*$. By Lemma \ref{Pako} each one of them is of the form $U_d = V_d \cap (V_{r_1}^\perp\cap\cdots 
\cap V_{r_l}^\perp)$, for some divisor $d$ of $m$ different from $1$ and $m$ and a complete set $\{r_1,\ldots,r_l\}$ of 
elements of $D(f)$ covered by $d$. Similarly as for $U_m$, now $U_d$ contains ${C}_{k}$, where $k$ is a multiple of 
$m/d$, but not of any prime factors of $d$. By duality, for the same $k$ the dual functions $C_k^*$ belong to a basis of solutions of $U_d^*$. Besides, that also means 
that $\Phi_d(z)$ divides $P_C(z)$.
\end{proof}

For some intermediate steps in the solution of the Tangential center problem we shall 
need to solve not the Tangential center problem but the more general equation
$$
\int_{C(t)}g=p(t),$$
where the righthand side is a polynomial $p(t)\in\C[t]$.

\begin{prop}\label{prop:cadena}
Let $p(t)\in\mathbb{C}[t]$, $p(t)\neq 0$. There exists a solution of $\int_{C(t)}g\equiv p(t)$ if and only if $C(t)$ is not a 
cycle.

Moreover, if $C(t)$ is not a cycle, then $\int_{C(t)}g\equiv p(t)$ if and only if 
\[
g=\dfrac{p\circ f}{\sum n_i}+u,
\]
where $u$ is a solution of
$\int_{C(t)}u\equiv 0$.
\end{prop}

\begin{proof}
Assume that  $\int_{C(t)} g \equiv p(t)$. Since $p(t)$ is invariant by the action of $G_f$, then 
\[
|G_f|\, p(t) = \sum_{\sigma\in G_f} \sigma(p(t)) = \sum_{\sigma\in G_f} \int_{\sigma(C(t))} g = 
\frac{|G_f|}{m} \sum_{k=1}^m \left( \sum_{i=1}^m n_i \right) g(z_k(t)).
\]
If $C(t)$ is a cycle, it follows $p(t)=0$, contrary to the assumption.

If $C(t)$ is not a cycle, then $\int_{C(t)}g\equiv p(t)$ is equivalent to 
\[
0\equiv \int_{C(t)}g-p(t)=\sum n_i \left( g(z_i(t))-\frac{p(f(z_i(t)))}{\sum n_i}\right)=\int_{C(t)} u.
\]
\end{proof}

\section{Solution for $2$-transitive, monomials and Chebyshev polynomials}

In this section we prove Theorem~\ref{theo:inductionsptep}, which 
we have divided into Propositions \ref{prop:xmT} and \ref{trantriv}.
That is, we shall solve the zero-dimensional tangential center problem
for the basis of induction, this is, when $f$ is a $2$-transitive polynomial, a monomial or a Chebyshev polynomial.
We need to consider chains instead of cycles for they appear in the induction process.

First assume that $f(z)=z^m$ or $f(z)=T_m(z)$, where $T_m=\cos(m\arccos(z))$ is the $m$-th Chebyshev polynomial, with $m$ not necessarily 
prime. Let $C(t)$ be a balanced cycle of $f$ (with real coefficents). Under the above assumptions we calculate $(V(t)^\perp)^*$  explicitly. 
The key point of the calculation resides in the fact that dual vectors $C_k^*$ of the chains $C_k$, 
are easily calculated in the monomial case. Similarly, in the Chebyshev case, 
the dual space $Vect(C_k,C_{m-k})^*$ of the inseparable space $Vect(C_k,C_{m-k})$ is easily calculated.

\begin{prop} \label{prop:xmT}
Let  $f(z)=z^m$ (resp. $f(z)=T_m(z)$) be a monomial function (resp. Chebyshev polynomial) and $C(t)$ a balanced chain of $f$. 
Then:
\begin{enumerate}
\item The $\mathbb{C}[t]$-module of solutions of the tangential center problem is given by
\[
U^*=U_{m}^*\oplus U_{d_1}^*\oplus\cdots\oplus U_{d_j}^*.
\]
where a basis of $U_1^*$ is given by $g_0(z)=1$ and a basis of $U_m^*$ is given by the functions $g_k(z)=z^k$ 
(resp. $g_k(z)=T_k(z)$) for $k$ relatively prime with $m$. For any $d_j$ divisor of $m$, a basis of $U_{d_j}^*$ is given by 
functions of the form $g_k(z)=z^k$ (resp. $g_k(z)=T_k(z)$) for $k$ a multiple of $m/d_j$, but not of any prime factors 
of $d_j$, i.e., $g.c.d.(m,k)=m/d_j$.

\item The above spaces $U_j^*$ are mutually orthogonal. Their dimensions are given by 
\[
dim_{\C} (U_1^*)=1, \quad dim_{\C} (U_m^*)=\phi(m), \quad dim_{\C} (U_{d_l}^*)=\phi(d_l).
\] 

\item The space of solutions of the tangential center problem is generated by 
$\{ g_j(z) = z^j \text{ (resp. } g_j(z)=T_j(z)) : \,\, \Phi_{m/g.c.d.(m,j)}(z) | P_C(z) \}$.
\end{enumerate}
\end{prop}

\begin{proof}
Let $f(z)=z^m$ or $f(z)=T_m(z)$, $C(t)$ a balanced chain of $f$ and $V(t)$ the orbit of $C(t)$ defined as above. Let $D(m)$ be 
the set of divisors of $m$. Note that for any divisor $k\in D(m)$ one has $f=g_{m/k} \circ g_k$, where $g_l(z)=z^l$ (resp. 
$g_l(z)=T_l(z))$. This shows that in the two particular cases for $f$ a monomial or a Chebyshev polynomial, 
the set $D(f)$ coincides with 
$D(m)$, since the decompositions of $f$ are given by the divisors of $m$. By Lemma \ref{general}, we now know the complete 
decomposition of the space of solutions in $G_f$-invariant spaces, which is the same in both cases. We have to calculate the 
dual spaces of each of the direct summands of $V(t)^{\perp}$.

Let $g_l(z)= z^l$ and $z_j(t)= t^{1/m} \epsilon_m^{j-1}$, $j=1,\ldots,m$, for $\epsilon_m$ a primitive $m$-th root of unity. 
Then 
\[
\begin{split}
\int_{C_k} g_l & = \sum_{j=1}^{m} \overline{w_{k,j}} g_l(z_j(t)) = \sum_{j=1}^{m} \overline{\epsilon_m^{k(j-1)}} t^{l/m} 
\epsilon_m^{l(j-1)} \\ & = t^{l/m} \sum_{j=1}^{m} \overline{\epsilon_m^{k(j-1)}} \epsilon_m^{l(j-1)} = t^{l/m} <w_l,w_k>\\ 
& = m t^{l/m} \delta_{lk}.
\end{split}
\]
This shows that 
\[
g_l=c_l C_l^*,
\] 
where $c_l$ is a nonzero constant. The claim in the monomial case follows now from Lemma \ref{general}
(note that the inequalities of the dimensions must be equalities in this case).

\medskip
Consider now the Chebyshev case. Let $T_m(z)= \cos(m \arccos (z))$ be the $m$-th Chebyshev polynomial, which is a polynomial 
of degree $m$. If we take $f(z)=T_m(z)$, then for any $t \in \C$, $f(z)=t$ gives $m$ preimages
\[
z_k(t) = (T_m^{-1}(t))_k = \cos \left( \dfrac{1}{m} \arccos_k(t) \right),  
\]
where we choose the range of $\arccos_k$ in $[0, 2\pi) + 2 (k-1) \pi$ (indeed $\arccos_k(t)=\arccos_1(k)+2(k-1)\pi$,
see Figure~\ref{fig:1}). Note that a loop of $t$ around infinity transforms $\arccos_k(t)$ into $\arccos_{k+1}(t)$
(see Figure~\ref{fig:2}). 

\begin{figure}[h]
\begin{center}
{\scalebox{.4}{\includegraphics{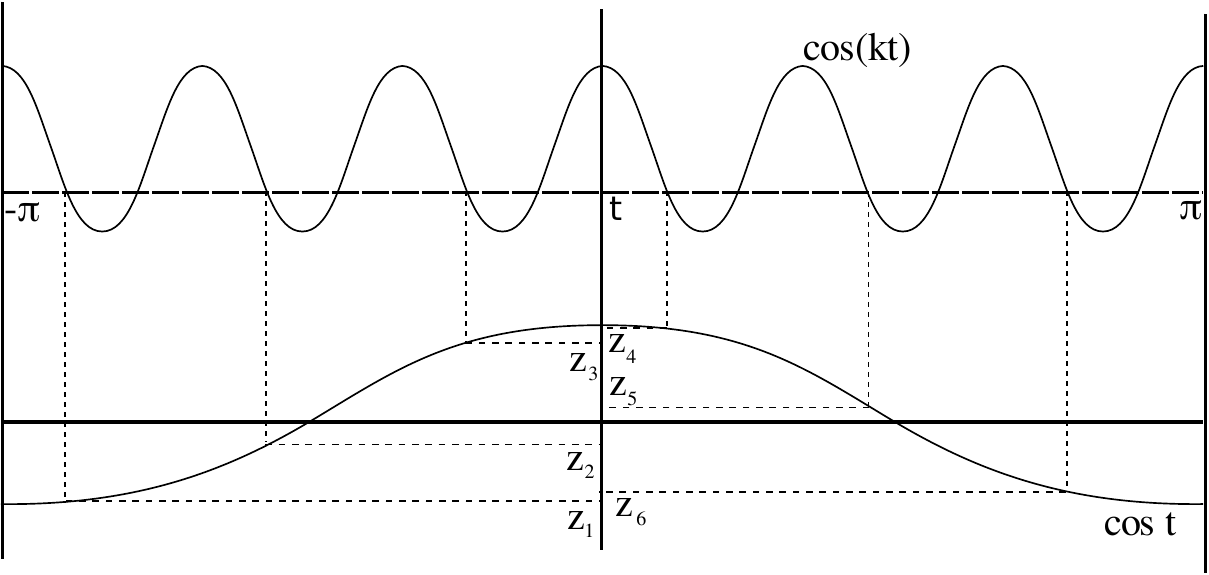}}}
\end{center}
\caption{Computation of $z_k(t)$}\label{fig:1}
\end{figure}

The Chebyshev polynomials $T_0,\ldots T_{m-1}$ form a basis of the space of polynomials as a $\C[t]$-module.

\begin{figure}[h]
\begin{center}
{\scalebox{.4}{\includegraphics{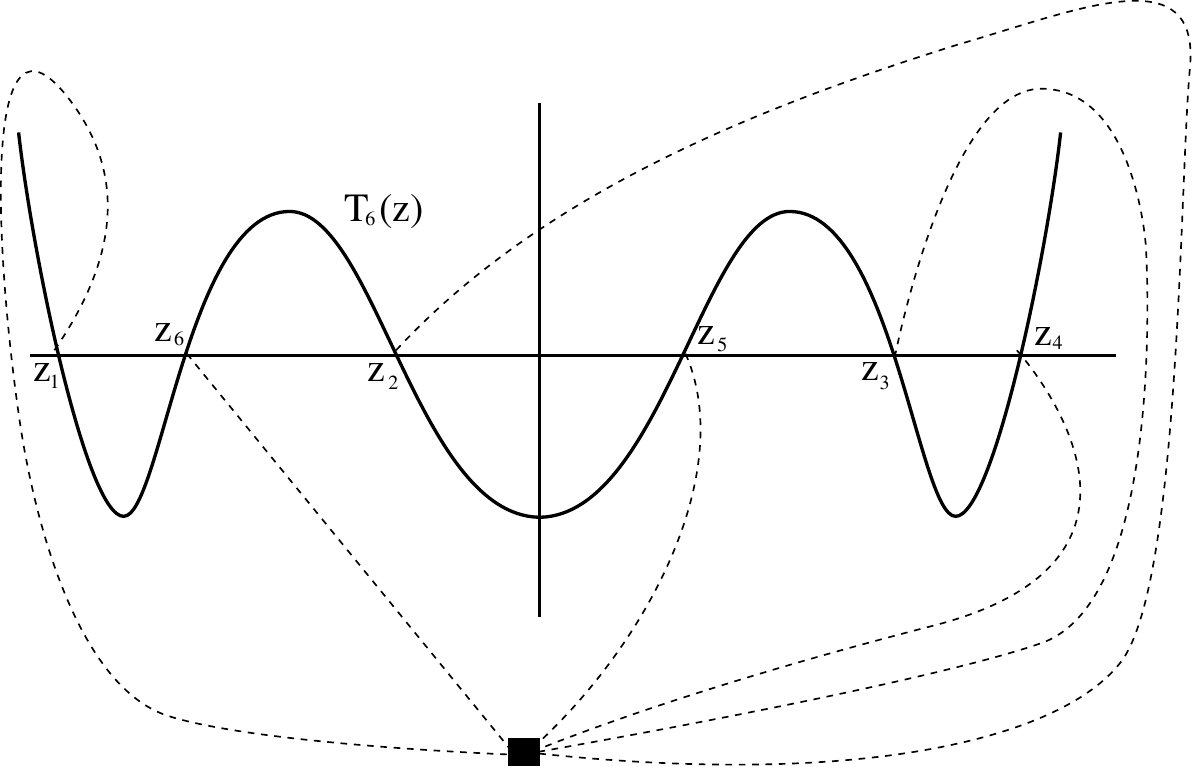}}}
\end{center}
\caption{Monodromic action of a rotation around infinity}\label{fig:2}
\end{figure}

If we take $t \in \R$ and denote $\xi(t) = \arccos_1(t) \in \R$,  $\epsilon_m = e^{i 2\pi/m}$, then for any chain $C(t)$ we 
obtain
\begin{equation}\label{intct}
\begin{split}
\int_{C(t)} T_j(z) & = \sum_{k=1}^m n_k T_j(z_k) = \sum_{k=1}^m n_k \cos\left( \dfrac{j}{m} \arccos_k(t) \right) \\ 
& =  \sum_{k=1}^m n_k \cos \left(  \xi(t) \dfrac{j}{m} + 2(k-1) \pi \dfrac{j}{m} \right) \\ 
& = \sum_{k=1}^m n_k \left(  \frac{e^{i \xi(t) \frac{j}{m}}}{2} e^{i2(k-1) \pi \frac{j}{m}} 
+ \frac{e^{-i \xi(t) \frac{j}{m}}}{2} e^{-i2(k-1) \pi \frac{j}{m}}  \right) \\
& = \frac{e^{i \xi(t) \frac{j}{m}}}{2} \left( \sum_{k=1}^m n_k \epsilon_m^{j(k-1)} \right) + 
\frac{e^{-i \xi(t) \frac{j}{m}}}{2} \left( \sum_{k=1}^m n_k \epsilon_m^{-j(k-1)} \right) \\
& = \frac{e^{i \xi(t) \frac{j}{m}}}{2} P_C (\epsilon_m^j) + 
\frac{e^{-i \xi(t) \frac{j}{m}}}{2} P_C(\overline{\epsilon_m}^j) \\
& = \alpha_j(t) P_C (\epsilon_m^j) + \overline{\alpha_j(t)} P_C(\overline{\epsilon_m}^j) ,\\
\end{split}
\end{equation}
where $\alpha_j(t) = \dfrac{e^{i \xi(t) \frac{j}{m}}}{2}$. If the chain $C(t)$ has real coefficients, this gives
\[
\int_{C(t)}T_j(z)= 2 Re( \alpha_j(t) P_C (\epsilon_m^j) ),
\]
due to the fact that $P_C(z) \in \R[z]$.

Consider in particular the chain $C_k$. Note that $P_{C_k}(\epsilon_m^j) = <w_j,w_k>$. Hence \eqref{intct} gives
\[
\int_{C_k(t)} T_j = \alpha_j(t) <w_j,w_k> + \overline{\alpha_j(t)} <w_{m-j},w_k>.
\]
It follows that $\int_{C_k(t)} T_j=0$ if $k \not \in \{j,m-j\}$.

On the other hand, as our balanced cycle $C(t)$ is real, it follows that the polynomial 
$P_C(z)$ has real coefficients and hence $<w_{m-j}, C> = P_C(\epsilon_m^{m-j}) = P_C(\epsilon_m^{-j}) = P_C(\overline{\epsilon_m^j}) = 
\overline{P_C(\epsilon_m^j)} = \overline{<w_j, C>}$. This means that $w_j\in V(t)^\perp$, if and only if $w_{m-j} \in V(t)^{\perp}$. It 
follows that the two-dimensional dual space of the space generated by $w_j$ and $w_{m-j}$ is the space generated by $T_j$ and $T_{m-j}$. 
The claim now follows as in the monomial case from Lemma \ref{general}.
\end{proof}

\begin{exam} 
Let $f(z)=z^6$. Then $D(f)=\{1,2,3,6\}$ and using Lemma \ref{Pako} we get 
$H_0(f^{-1}(t)) \equiv \C^6 = U_1 \oplus U_2 \oplus U_3 \oplus U_6$, 
where $U_1 = <C_0>$, $U_2 = <C_3>$, $U_3 = <C_2,C_4>$, $U_6 = <C_1,C_5>$. 
Let $C$ be a cycle of $f$, $V$ its orbit and $V^\perp$ the orthogonal complement. 
By the cycle condition $U_1 \subset V^\perp$. 
 
Assume that the cycle $C(t)$ is balanced. This is equivalent to assuming 
$U_6 \subset V^\perp$. Now various balanced cycles can be considered. 
For instance, if $C$ is such that $V=U_2$, then the space of solutions $g$ as a 
$\mathbb{C}[t]$-module is generated by $\{1,z, z^2,z^4,z^5\}$. If $V=U_3$, then 
the solution is generated by $\{1,z,z^3,z^5\}$.
\end{exam}

By the Burnside-Schur Theorem, if $f$ is primitive and it is not in the cases above, then it is $2$-transitive. The unbalanced case is solved 
in Theorem \ref{theo:abm}; therefore only the balanced case remains. The following result solves the tangential center problem for a balanced 
cycle and a $2$-transitive $f$.

\begin{prop}\label{trantriv}
Let $f\in\C[z]$ be a polynomial with a $2$-transitive monodromy group and $C(t)=\sum_{j=1}^m n_jz_j(t)$ a chain of $f$.
\begin{enumerate}
\item If $C(t)$ is a balanced chain, then there exists $n$ such that $n_j=n$ for every $1\leq j\leq m$. 
In particular if $C(t)$ is a balanced cycle, then it is trivial.
\item If $C(t)$ is a balanced chain, then $\int_{C(t)}g\equiv 0$  if and only if
\[
\sum_{k=1}^{deg(g)} s_k g_k = 0,
\]
where  
$g(z)=\sum g_k z^k$ and $s_k$ are given by the Newton-Girard formulae, as
\[
s_k=\sum_{i=1}^{m} z_i^k(t).
\]
\end{enumerate} 
\end{prop}

\begin{proof}
(1) Let us assume that $C(t)$ is balanced. It means that $\sum_{j=1}^m n_{\sigma(j)}\epsilon_m^j=0$, for any $\sigma\in G_f$. Let $H_1 = \{ \sigma 
\in G_f| \sigma(1)=1\}$ be the stabilizer of $z_1(t)$. Then 
\[
\sum_{\sigma\in H_1}\sum_{j=1}^m n_{\sigma(j)}\epsilon_m^j=0.
\]
That is,
\begin{equation}\label{sum}
|H_1| n_1 \epsilon_m + \sum_{j=2}^m \sum_{\sigma\in H_1}n_{\sigma(j)}\epsilon_m^j=0.
\end{equation}
Now the assumption that $G_f$ is $2$-transitive and $H_1$ is the stabilizer of $z_1(t)$ implies that $H_1$ acts transitively on $z_2(t),\ldots,z_m(t)$. Hence, for each $j\in\{2,\ldots,m\}$ and each $k\in\{2,\ldots,m\}$ there is the same number of occurrences of $n_k$ in the sum $\sum_{\sigma\in H_1}n_{\sigma(j)}\epsilon_m^j$. This number is $\frac{|H_1|}{m-1}$ and 
\[
\sum_{j=2}^m \sum_{\sigma\in H_1} n_{\sigma(j)} \epsilon_m^j =
\frac{|H_1|}{m-1} \sum_{k=2}^m n_k \sum_{j=2}^m \epsilon_m^j=
\frac{|H_1|}{m-1} \left(-n_1 + \sum_{k=1}^m n_k \right) \sum_{j=2}^m \epsilon_m^j.
\] 
Observing that $\sum_{j=2}^m\epsilon_m^j=-\epsilon_m$, by Equation~(\ref{sum}) we get
\[
\begin{split}
0 & = |H_1| n_1 \epsilon_m - \frac{|H_1|}{m-1} \epsilon_m \left(-n_1 + \sum_{k=1}^m n_k\right)\\ 
& = |H_1| \epsilon_m \left( n_1 + \frac{n_1}{m-1} - \frac{\sum_{k=1}^m n_k}{m-1} \right)\\ 
& = \frac{|H_1|}{m-1} \epsilon_m \left( m n_1 - \sum_{k=1}^m n_k \right).
\end{split}
\]
Since $|H_1| \ne 0$ and $\epsilon_m \ne 0$, then 
\[
n_1 = \frac{\sum_{k=1}^m n_k}{m}.
\]

The choice of index $1$ was arbitrary. Hence
 $n_j$ is a constant not depending on $j$. If 
$C(t)$ is a cycle, then this constant must be zero.

(2) Let us compute the solutions of $\int_{C(t)}g\equiv 0$, with $g(z)=\sum g_k z^k$, $g_k \in \C$. By (1), we can assume that $C(t)=n\sum z_i(t)$.
Then
\[
\int_{C(t)} g = n \sum_{i=1}^m g(z_i(t)) = n \sum_{i=1}^m \sum_{k=1}^{\deg(g)} g_k z_i^k(t) = n \sum_{k=1}^{\deg(g)} g_k \sum_{i=1}^m z_i^k(t) 
\equiv 0,
\]
and the solution follows from the Newton-Girard formulae.
Recall that Newton-Girard formulae express explicitly $s_k(t)=\sum z_i^k(t)$ in function of the coefficients of the polynomial $f(z)-t$.
\end{proof}

\section{Monodromy group of imprimitive polynomials}\label{sec:mon}

In the next section we will prove Theorems~\ref{theo:main} and \ref{theo:main2}. We shall assume that Hypothesis~\ref{hyp} holds, which,
as we prove in this section, will allow 
us to write the monodromy group of $f$ as a semidirect product defined by the monodromy groups of $\tilde f$ and $h$.

First, we introduce a new numbering in the preimages of $t$ by $f$. Fix a regular value $t$ of $f$ and take the preimage by 
$\tilde f$ of $t$. We obtain $m/d$ points, $w_1(t),\ldots,w_{m/d}(t)$. For each $w_i(t)$, let $z_{i,j}(t)$, $j=1,\ldots d$ 
denote each of the preimages of $w_i(t)$ by $h$.

Then, according to what we saw in Section~\ref{Section:1}, the blocks of the imprimitivity system associated to $f= \tilde f 
\circ h$ are $B_i = \{ z_{i,j}(t): j=1,\ldots,d \}$, $i=1,\ldots,m/d$. They correspond to rows of circles in Figure~\ref{mono}.

Differentiating $f=\tilde f \circ h$, we obtain $f'(z)=\tilde f'(h(z))h'(z)$. Thus, critical points of $f$ correspond to 
either the preimage by $h$ of critical points of $\tilde f$ or critical points of $h$. Let us denote 
\[
\{z\colon f'(z)=0\}=\{a_1,\ldots,a_{d(m/d-1)},b_1,\ldots,b_{d-1}\},
\]
where $\tilde f'(h(a_i))=0$, $h'(b_i)=0$. 

Let $\alpha_i$ denote the permutation associated to $f(a_i)$ and $\beta_i$ the permutation associated to $f(b_i)$. Each 
permutation $\alpha_i$ (resp. $\beta_i$) corresponds to winding counter-clockwise around only one critical value $f(a_i)$ 
(resp. $f(b_i)$) along a closed path.

\begin{figure}[ht]
\begin{center}
{\scalebox{.15}{\includegraphics{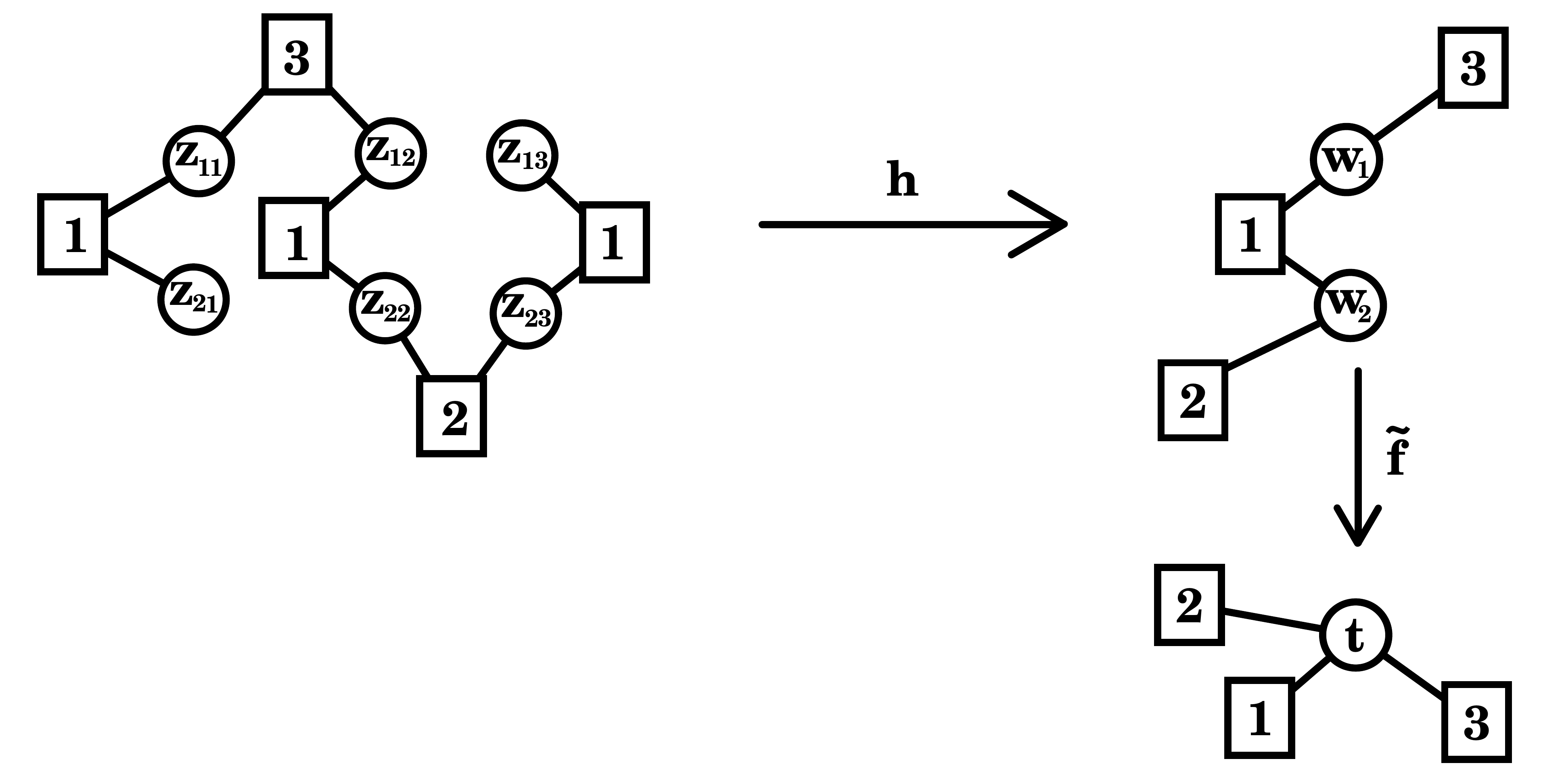}}}
\end{center}
\caption{Regular and critical points/values of imprimitive polynomials}\label{mono}\end{figure}

Note that by assumption the paths giving $\beta_i$ lift to closed paths (loops) based at whatever $w_i(t)$ we take as starting 
point as they encircle no critical value of $\tilde f$. Similar claim is valid for $\alpha_i$. This gives 
\begin{equation}\label{ab}
\alpha_k(z_{i,j})=z_{\alpha_k(i),j},\quad  \beta_k(z_{i,j})=z_{i,\beta_k(i,j)},
\end{equation}
so that $\alpha_k$ exchange whole blocks, while $\beta_k$ only moves elements inside every block.

Thanks to Hypothesis \ref{hyp}, we can split the monodromy group $G_f$ of $f$ in terms of the monodromy group of $\tilde f$ 
and a normal subgroup of $G_f$ which contains (a subgroup isomorphic to) the monodromy group of $h$.

\begin{lema}\label{lema:semidirect}
Assume that Hypothesis~\ref{hyp} holds. Let us denote 
\[
G_{\tilde f}=<\alpha_k>,\quad N_h=<\alpha_i \beta_k \alpha_i^{-1}\colon i,k>.
\] 
Then $G_f$ is the semidirect product $N_h \rtimes_{\phi} G_{\tilde f}$ of $N_h$ and $G_{\tilde f}$ with respect to $\phi: G_{\tilde f} \to Aut(N_h)$, $\phi(\alpha)=\phi_{\alpha}$, where $\phi_{\alpha}(\sigma)= \alpha \sigma \alpha^{-1}$.
\end{lema}

\begin{proof}
First, we define the group semidirect product $N_h \rtimes_{\phi} G_{\tilde f}$ as the cartesian product set $N_h \times G_{\tilde f}$ with 
the following operation defined by $\phi$:
\[
(\sigma, \alpha) (\tilde{\sigma}, \tilde{\alpha}) = ( \sigma \phi_{\alpha} (\tilde{\sigma}) , \alpha \tilde{\alpha} ) = 
(\sigma \alpha \tilde{\sigma} \alpha^{-1},  \alpha \tilde{\alpha}).
\]
Obviously $N_h \cap G_{\tilde f} = \{ Id \}$. Hence, if we prove that any element $\sigma \in G_f$ is written in the form $\sigma = \tau 
\alpha$, where $\tau \in N_h$ and $\alpha \in G_{\tilde f}$, then this decomposition is unique (if $\tau \alpha = \tilde{\tau} \tilde{\alpha}$, 
then $\tilde{\tau}^{-1} \tau = \tilde{\alpha} \alpha^{-1} \in N_h \cap G_{\tilde f} = \{ Id \}$ and $\tau = \tilde{\tau}$, $\alpha = 
\tilde{\alpha}$) and we will have the group isomorphism
\[
\begin{matrix}
G_f & \longrightarrow & N_h \rtimes G_{\tilde f} \\
\sigma = \tau \alpha & \mapsto & (\tau, \alpha)
\end{matrix}
\]

Since $G_f$ is generated by the permutations $\{ \alpha_i, \beta_j: i=1, \ldots,d(m/d -1), j=1, \ldots, d-1 \}$, every element $\sigma$ in 
$G_f$ is a finite product of $\alpha$'s and $\beta$'s. In order to write $\sigma$ as the product of a permutation $\tau$ in $N_h$ and a 
permutation $\alpha$ in $G_{\tilde f}$, we group together the $\alpha$'s and $\beta$'s that appear in the expression of $\sigma$ in the 
following way:
\[
\sigma = \prod_{i \in I^{\beta}_1} \beta_i \prod _{j \in I^{\alpha}_2} \alpha_j \prod_{k \in I^{\beta}_3} \beta_k \ldots 
\prod _{l \in I^{\alpha}_r} \alpha_l, 
\]
where the set of indices $I^{\beta}$ are permutations with repetitions of $\{ 1,\ldots,d-1 \}$ and the set of indices $I^{\alpha}$ are permutations with repetitions of $\{1,\ldots,d(m/d-1) \}$. The first product belongs to $N_h$, but the second product $\prod _{j \in I^{\alpha}_2} \alpha_j$ is in $G_{\tilde f}$, so we multiply the third product $\prod_{k \in I^{\beta}_3} \beta_k$ on the right by $ (\prod _{j \in I^{\alpha}_2} 
\alpha_j)^{-1} \prod _{j \in I^{\alpha}_2} \alpha_j = Id$. Now we rewrite $\sigma$ as follows
\[
\sigma = \prod_{i \in I^{\beta}_1} \beta_i \left( \prod _{j \in I^{\alpha}_2} \alpha_j \prod_{k \in I^{\beta}_3} \beta_k 
(\prod _{j \in I^{\alpha}_2} \alpha_j)^{-1} \right) \prod _{j \in I^{\alpha}_2} \alpha_j \ldots \prod _{l \in I^{\alpha}_r} \alpha_l,
\]
where the first and the second products $ \prod_{i \in I^{\beta}_1} \beta_i, \left( \prod _{j \in I^{\alpha}_2} \alpha_j \prod_{k \in I^{\beta}_3} \beta_k (\prod _{j \in I^{\alpha}_2} \alpha_j)^{-1} \right)$ are in $N_h$.

We follow this procedure with the new third product until there is no product of $\beta$'s left, and in a finite number of times we have $\sigma$ rewritten as a product of several permutations in $N_h$ and a final one in $G_{\tilde f}$, which can be renamed as $\tau \in N_h$ and $\alpha \in G_{\tilde f}$. 
\end{proof}

\begin{lema}\label{lema:balanced}
Assume that $f= \tilde f \circ h$ satisfies Hypothesis~\ref{hyp}. Let $H_{i_0,j_0}$ denote the stabilizer of $(i_0,j_0)$ in $N_h$. Then: 
\begin{enumerate}
\item $G_{\tilde f}$ is isomorphic to the monodromy group of $\tilde f$.
\item $N_h$ is a normal subgroup of $G_f$ such that for every block $B_i$ of $f$ there exists a subgroup $G_i$ of $N_h$ such that $G_i$ leaves fixed all the elements $(j,k)\not \in B_i$. Moreover, $G_i$ is isomorphic to the monodromy group $G_h$ of $h$.
\item The subgroup $\bigcap_{j=1,\ldots,d} H_{i_0,j}\subset N_h$ acts transitively on the elements of $B_i$ for every $i\neq i_0$.
\item If $h$ is $2$-transitive, then $H_{i_0,j_0}$ acts transitively on the other elements of the block $B_{i_0}$, thus, on $\{(i_0,j)\colon j\neq j_0\}$.
\end{enumerate}
\end{lema}

\begin{proof}
$(1)$ and $(4)$  follow easily and (3) is a consequence of $(2)$, since $\bigcap_{j=1,\ldots,d} H_{i_0,j}$ contains $G_i$ for every $i \neq i_0$.

To conclude, we prove (2). The group $N_h$ is generated by elements of the form $\tau=\alpha\beta\alpha^{-1}$, with $\alpha\in G_{\tilde f}$ and $\beta\in G_h$. By Hypothesis~\ref{hyp}, $\beta$ moves elements of at most one block $B_i$. Since the group $G_{\tilde f}$ acts transitively on each column, then for every $B_i$ there exists a subgroup of $N_h$, which we will call $G_i$, isomorphic to $G_h$ that moves only the elements of $B_i$.
\end{proof}

Now, we shall prove that the non-merging hypothesis on critical values implies that for any two decompositions of a polynomial, one of the inner factors 
factorizes by the other one. We recall that $\tau_\infty=(1,2,\ldots,m)\in G_f$.

\begin{prop}\label{prop:uniqueness}
Assume that $f= f_0 \circ h_0$ satisfies Hypothesis~\ref{hyp}, $f,f_0,h_0\in\mathbb{C}[z]$. 
If there exist $f_1,h_1\in\mathbb{C}[z]$ such that $f=f_1\circ h_1$, then 
there exists a polynomial $w$ such that either $h_1=w\circ h_0$ 
or $h_0=w\circ h_1$.
\end{prop}
\begin{proof}
Let $B_0$ be a block associated to the decomposition $f=f_0\circ h_0$ such that $1\in B_0$, and 
let $B_1$ a block associated to  $f=f_1\circ h_1$ such that $1 \in B_1$. Then either $B_1\subset B_0$ or there exists 
$i\in B_1\backslash B_0$.

If $B_1\subset B_0$, then there exists $w\in\C[z]$ such that $w(h_1)=h_0$: we know that $B_0= \{ j \in \{ 1,\ldots,m \}: 
h_0(z_1)=h_0(z_j)\}$ and $B_1= \{ k \in \{ 1,\ldots,m \}: h_1(z_1)=h_1(z_k)\}$. Their stabilizers $G_{B_0}= \{ \sigma \in 
G_f: \sigma(B_0)=B_0\}$ and $G_{B_1}= \{ \sigma \in G_f: \sigma(B_1)=B_1\}$ verify that $G_{B_1} \subset G_{B_0}$ and, 
therefore, $L^{G_{B_0}}= \C(h_0(z_1)) \subset L^{G_{B_1}} = \C(h_1(z_1))$, where $L= \C(z_1(t),\ldots,z_m(t))$
and $L^{G_k}$ denotes the elements of $L$ invariants by the action of $G_k$.  As a 
consequence, there exists a rational function $w$, which we can assume polynomial, such that $h_0 = w \circ h_1$.

If there exists $i\in B_1\backslash B_0$, then $\tau_\infty^{i-1}(B_0)\subset B_1$, and there exists $w\in\C[z]$ such that 
$w(h_0)=h_1$: since $\tau_{\infty}^{i-1}(1)=i \in \tau_{\infty}^{i-1}(B_0)$, this one is a block of the imprimitivity system 
associated to the decomposition $f=f_0 \circ h_0$ different from $B_0$, then $1 \not\in \tau_{\infty}^{i-1}(B_0)$. By Lemma 
\ref{lema:balanced} (3), we can assure that the stabilizer $H_1$ (in $G_f$) of $1$ is transitive on every other 
block different from $B_0$. Therefore, $H_1$ is transitive on $\tau_{\infty}^{i-1}(B_0)$ and the orbit of $i$ by $H_1$ 
contains $\tau_{\infty}^{i-1}(B_0)$. Since every permutation in $H_1$ leaves $B_1$ fixed and $i \in B_1$, $\tau_{\infty}^{i-1}(B_0) 
\subset B_1$ and following a similar argument as in the previous case, there exist $w \in \C[z]$ such that $h_1= w \circ h_0$. 
\end{proof}

Proposition~\ref{prop:uniqueness} implies that when we have chains of decompositions satisfying Hypothesis~\ref{hyp}, 
then we have uniqueness.

\begin{proof}[Proof of Proposition \ref{unicity}]
We shall prove it by induction on the number $\tilde d$ of factors in the decomposition of $f$.

First, assume that $\tilde d=0$ and $f=\tilde f_0$ is $2$-transitive, Chebyshev or a monomial, and let $f=f_0\circ\ldots\circ  f_d$ be a 
decomposition of $f$ for some polynomials $f_0,\ldots,f_d$ such that every $f_k$ is either $2$-transitive, Chebyshev or a monomial,
and that Hypothesis~\ref{hyp} holds for every $ f_0 \circ \ldots \circ f_k = \Big( f_0 \circ \ldots \circ f_{k-1} \Big) \circ 
f_k$. Since 2-transitive polynomials are primitive and monomials and Chebyshev polynomials 
can not decompose satisfying Hypothesis~\ref{hyp}, then $d=0$.

Now, assume that $f=f_0\circ\ldots\circ f_d$ for some polynomials $f_0,\ldots,f_d$ such that every $f_k$ is either $2$-transitive, 
Chebyshev or a monomial, and that Hypothesis~\ref{hyp} holds for every $f_0 \circ \ldots \circ f_k = \Big( f_0 \circ \ldots 
\circ f_{k-1} \Big) \circ f_k$.

If  $f=\tilde f_0\circ\ldots\circ \tilde f_{\tilde d}$ is another decomposition of $f$ for some polynomials $\tilde f_0,
\ldots,\tilde f_d$ such that every $\tilde f_k$ is either $2$-transitive, Chebyshev or a monomial, by 
Proposition~\ref{prop:uniqueness} there exists $w\in\C[z]$ such that either $f_d=w\circ \tilde f_{\tilde d}$ or 
$\tilde f_{\tilde d}=w\circ f_{d}$.

Assume that $\tilde f_{\tilde d}=w\circ f_{d}$ for some non-linear $w$. In particular
this implies that $w,f_d$ are both monomials or both Chebyshev polynomials.
Then 
\[
f_0\circ\ldots\circ f_{d-1}=\tilde f_0\circ\ldots\circ \tilde f_{\tilde d-1}\circ w.
\]
Thus, Hypothesis~\ref{hyp} does not hold, since $\tilde f_0\circ\ldots\circ \tilde f_{\tilde d-1}\circ w$
and $f_d$ share a critical value. 

Therefore, we may assume that $f_d=w\circ \tilde f_{\tilde d}$. But in this case
\[
f_0\circ\ldots\circ f_{d-1}\circ w=\tilde f_0\circ\ldots\circ \tilde f_{\tilde d-1}.
\]
Note that $f_0\circ\ldots\circ f_{d-1}\circ w$ still satisfies Hypothesis~\ref{hyp} since the critical values of $w$ are a 
subset of those of $f_d$. We conclude by induction.

If Hypothesis~\ref{hyp} holds for every $\tilde f_0 \circ \ldots \circ \tilde f_k = \Big( \tilde f_0 \circ \ldots \circ 
\tilde f_{k-1}\Big)\circ \tilde f_k$, interchanging $\{f_k\}$ and $\{\tilde f_k\}$ in the previous arguments we conclude.
\end{proof}

\section{Proof of Theorems~\ref{theo:main} and \ref{theo:main2}}\label{section:demo}

Let $C(t)$ be a balanced cycle of an imprimitive polynomial $f=\tilde f \circ h$, with $h$ $2$-transitive, Chebyshev or a monomial. According to the type of $h$, we deduce first some information on how the cycle $C(t)$ is positioned with 
respect to the imprimitivity system defined by $h$.

\begin{prop}\label{prop:balanced}
Assume that $f=\tilde f \circ h$ satisfies Hypothesis \ref{hyp}. Let $C(t)$ be a balanced cycle of $f$ and let $\mathcal{B}_h = \{B_1,\ldots,B_{m/d}\}$ denote the
 imprimitivity system corresponding to $h$.
\begin{enumerate}
\item If $h$ is $2$-transitive, then $n_k$ is constant for $k\in B_j$, for every $B_j\in\mathcal{B}_h$.

\item If $h(z)=z^d$ or $h(z)=T_d(z)$, then the restriction of $C(t)$ to each block of $\mathcal{B}_h$ is balanced.
\end{enumerate}
\end{prop}

\begin{proof}
$(1)$ 
Let us fix $i_0,j_0$ and let us denote $H_0=H_{i_0,j_0}\subset G_f$ the stabilizer of $z_{i_0,j_0}$. For every preimage $z_{i,j}(t)$ of $t$ by $f$, let us denote by $k(i,j)$ its position when we enumerate them so that $\tau_{\infty}=(z_1(t), z_2(t), \ldots, z_m(t))$. Besides, $z_{i,j}(t)= z_{k(i,j)}(t)$.

Therefore, since $C(t)$ is balanced,
\[
\begin{split}
0=&\sum_{\tau\in H_0} \sum_{i,j} n_{\tau(i,j)} \epsilon_m^{k(i,j)}\\
=&|H_0| n_{i_0,j_0} \epsilon_m^{k(i_0,j_0)}
+ |H_0| \left( \frac{\sum_{j=1,j\neq j_0}^{d} n_{i_0,j}}{d-1}\sum_{j=1,j\neq j_0}^{d} \epsilon_m^{k(i_0,j)} \right) +\\
& |H_0| \left( \sum_{i=1,i\neq i_0}^{m/d} \frac{\sum_{j=1}^{d} n_{i,j}}{d}\sum_{j=1}^{d} \epsilon_m^{k(i,j)} \right),
\end{split}
\]
where the second summand is due to Lemma~\ref{lema:balanced} (4), since $h$ is $2$-transitive, and the last summand is due to Lemma~\ref{lema:balanced} (3). Now, observe that $\sum_{j=1}^{d} \epsilon_m^{k(i,j)}=0$ for every $i$, since it is the sum of all the powers of $\epsilon_m$ corresponding to a block $B_i$, whose elements are the residue class mod $m/d$ for some $l \in \{ 1,\ldots,d  \}$ (see, for instance, \cite[Lemma 3.1]{PM}), that is, $\sum_{j=1}^d \epsilon_m^{k(i,j)} = \sum_{j=1}^{d} \epsilon_m^{l +jm/d} = \epsilon_m^l \sum_{j=1}^d \epsilon_m^{jm/d} = \epsilon_m^l \sum_{j=1}^d \epsilon_d^j =0$. Therefore,
\[
\begin{split}
0 & = |H_0| \left( n_{i_0,j_0} \epsilon_m^{k(i_0,j_0)} + \frac{\sum_{j=1,j\neq j_0}^{d} n_{i_0,j}}{d-1} (- \epsilon_m^{k(i_0,j_0)})  \right) \\ & = |H_0| \epsilon_m^{k(i_0,j_0)} \left( n_{i_0,j_0}(1+1/(d-1)) -\frac{\sum_{j=1}^{d} n_{i_0,j}}{d-1}\right). 
\end{split}
\]
Thus, $n_{i_0,j_0}$ does not deppend on $j_0$. Therefore $n_{i_0,j}$ is constant in the block $B_{i_0}$. 

$(2)$ Assume now that $h(z)=z^d$, or $h(z)=T_d(z)$. Let $H_0=\bigcap_{j=1,\ldots,d} H_{i_0,j}$. The elements $\{(i,j)\}_{j=1,\ldots,\deg(\tilde f)}$ with $i\neq i_0$ fixed are moved transitively by Lemma~\ref{lema:balanced} (3). Therefore, using similar arguments as in (1),
\[
0=\sum_{\tau\in H_0} \sum_{i,j} n_{\tau(i,j)} \epsilon_m^{k(i,j)} = |H_0| \sum_{j=1}^{d} n_{i_0,j} \epsilon_m^{k(i_0,j)}. 
\]
Then, $C(t)$ restricted to the block $B_{i_0}$ (the $i_0$-th $h$-invariant part of $C(t)$) is balanced. 
\end{proof}

\begin{proof}[Proof of Theorem~\ref{theo:main}]
By Proposition~\ref{prop:balanced}, $C(t)$ is constant along the blocks of the imprimitivity system corresponding to $h$, 
$\mathcal{B}_h = \{ B_1, \ldots, B_{m/d}\}$. Therefore,
\[
\int_{C(t)}g =\sum_i n_i g(z_i(t))=\sum_{i=1}^{m/d} n_i \sum_{k \in B_i} g(z_k(t)).
\]
Let us observe that $\sum_{k \in B_i} g(z_k(t))$ is invariant by $G_{B_i} = \{ \sigma \in G_f : \sigma(B_i)=B_i \}$. Then, 
if we denote $L = \C(z_1(t),\ldots,z_m(t))$, $\sum_{k \in B_i} g(z_k(t)) \in L^{G_{B_i}} = \C(h(z_k(t))$ for every $k \in 
B_i$. Let us denote $w_i(t) = h(z_k(t))$, which is constant in $k \in B_i$; then $w_i(t)$ is one of the $m/d$ preimages of $t$ by $\tilde f$. Therefore, $\sum_{k \in B_i} g(z_k(t)) = \tilde g_i(w_i(t))$ for some rational function $\tilde g_i(z)$, which we can assume is a polynomial. Since every subgroup $G_{B_i}$ of $G_f$ is conjugated to another $G_{B_j}$, it can easily be proved that $\tilde g_i$ does not depend on the block $B_i$, that is, there is only one polynomial $\tilde g(z)$ such that 
\begin{equation}\label{gtilde} 
\sum_{k \in B_i} g(z_k(t)) = \tilde g(w_i(t)), \quad i=1, \ldots,m/d.
\end{equation}
Then, 
since $h(C(t)) = \sum_{i=1}^{m/d} dn_i w_i(t)$,
\[
\int_{C(t)} g = \sum_{i=1}^{m/d} n_i \tilde g(w_i(t)) = %\int_{C_{f_0}(t)} \tilde g,
\frac{\int_{h(C(t))}  \tilde g}{d}.
\]

Since $z_k(t)$ for $k \in B_i$ are the preimages of $w_i$ by $h$, Equation~(\ref{gtilde}) can be rewritten as
\[
\sum_{k \in B} g(z_k(w)) = \tilde g(w).
\]
Replacing $g$ by its linear expansion in terms of the basis $\{1,z,\ldots, z^{d-1}, h(z), zh(z),$ 
$ \ldots, z^{d-1}h(z), h^2(z), \ldots\}$  of $\C[z]$, namely $g(z)= \sum_{i=0}^{d-1} z^i g_i(h(z))$, we get
\[
\begin{split}
\tilde g(w) & = \sum_{k \in B} \sum_{i=0}^{d-1} z_k^i(w) g_i(h(z_k(w))) = \sum_{k \in B} \sum_{i=0}^{d-1} z_k^i(w) g_i(w) \\ 
& = \sum_{i=0}^{d-1} \left( \sum_{k \in B} z_k^i(w) \right) g_i(w) = \sum_{i=0}^{d-1} s_i g_i(w),
\end{split}
\]
where $s_i$ can be obtained by the Newton-Girard formulae applied to $h$.
\end{proof}

\begin{proof}[Proof of Theorem \ref{theo:main2} ]
Assume that $f=\tilde f \circ h$, $h(z)=z^d$ or $h(z)=T_d(z)$
and Hypothesis \ref{hyp} is verified. Let $C(t)$ be a balanced cycle of $f$. Assume that
\[
\int_{C(t)} g\equiv 0,
\]
that is, $\sum_{i,j} n_{i,j} g(z_{i,j}(t))=0$. Fix $i_0\in\{1,\ldots,m/d\}$. Let $H_{i_0}$ be the stabilizer of 
$(i_0,j), j=1,\ldots,d$ in $N_h$. By Lemma \ref{lema:balanced}~(3), the elements $(i,j), j=1,\ldots,d$ with $i\neq i_0$ are moved 
transitively by $H_{i_0}$. Therefore,
\[
\begin{split}
0=&\sum_{\tau\in H_{i_0}} \sum_{i,j} n_{\tau(i,j)} g(z_{i,j}(t))\\
= & |H_{i_0}| \sum_{j=1}^d n_{i_0,j} g(z_{i_0,j}(t)) + 
\frac{|H_{i_0}|}{d-1} \sum_{i=1, i \neq i_0}^{m/d} \left( \sum_{j=1}^d n_{i,j} \right) \left( \sum_{j=1}^d g(z_{i,j}(t)) \right). 
\end{split}
\]
Let us observe that $ \sum_{j=1}^d g(z_{i,j}(t)) $ is invariant by $G_{B_i} = \{ \sigma \in G_f: \sigma(B_i) =B_i \}$, for 
$i \neq i_0$. As a consequence, if we denote $L=\C(z_1(t),\ldots,z_m(t))$, we obtain that $\sum_{j=1}^d g(z_{i,j}(t)) 
\in L^{G_{B_i}} = \C(h(z_{i,j}(t))) = \C(w_i(t))$, where $w_i(t) = h(z_{i,j}(t))$ for $j=1,\ldots,d$. 
By Lüroth Theorem, there exists a polynomial $\tilde{g} (w)$ (which arguing as in the previous proof does not deppend on $i$) such that 
\[
\sum_{j=1}^d g(z_{i,j}(t)) = \tilde{g} (w_i(t))
\]
and, consequently,
\[
\sum_{j=1}^{d} n_{i_0,j} g(z_{i_0,j}(t))=\frac{-1}{d-1}
\sum_{i=1, i \neq i_0}^{m/d} \left( \sum_{j=1}^d n_{i,j} \right) \tilde{g}(w_{i}(t)).
\] 
Replacing in the equation $\int_{C(t)}g=0$,
\[
\begin{split}
0=\sum_{i=1}^{m/d} \sum_{j=1}^d n_{i,j} g(z_{i,j}(t))=&
\frac{-1}{d-1}\sum_{i_0=1}^{m/d}\sum_{i=1, i \neq i_0}^{m/d} \left( \sum_{j=1}^d n_{i,j} \right) \tilde{g}(w_{i}(t))\\
=& - \frac{\frac{m}{d}-1}{d-1} \sum_{i=1}^{m/d} \left( \sum_{j=1}^d n_{i,j} \right) \tilde{g}(w_{i}(t)).
\end{split}
\]
Thus,
\begin{equation}\label{tildeg2}
\int_{h(C(t))}\tilde{g}=0.
\end{equation}

On the other hand,
\[
\begin{split}
\sum_{j=1}^{d} n_{i_0,j} g(z_{i_0,j}(t))&=\frac{-1}{d-1}
\sum_{i=1, i \neq i_0}^{m/d} \left( \sum_{j=1}^d n_{i,j} \right) \tilde{g}(w_{i}(t))\\
&=\frac{1}{d-1} \left( \sum_{j=1}^d n_{i_0,j} \right) \tilde{g}(w_{i_0}(t)).
\end{split}
\] 
In consequence, putting

\begin{equation}\label{pi}
p_{i_0}(w):=\frac{\sum_{j \in B_{i_0}} n_j }{d-1} \tilde{g}(w)\in \C[w]
\end{equation}
for every $i_0$, $g$ is a solution of 
\begin{equation}\label{g0}
\int_{\tilde C_{i_0}(w)} g =p_{i_0}(w), \quad 1\leq i_0\leq m/d.
\end{equation}

\begin{comment}
if and only if there exists $\tilde g$ such that 
\[
\int_{h(C(t))} \tilde g\equiv 0,\quad\int_{C_k(t)} g=\frac{\sum_{i\in B_k} n_i}{d-1}\tilde g(w_k(t)), \quad 
1\leq k\leq m/d, 
\]
where $C_k(t)=\sum_{i\in B_k} n_i z_i(t)$ is balanced, $\mathcal{B}_h=\{B_1,\ldots,B_{m/d}\}$ and $w_k(t)=h(z_i(t))$ 
for every $i\in B_k$.

Observe that 
\[
\int_{C_k(t)} g=\frac{\sum_{i\in B_k} n_i}{d-1}\tilde g(w_k(t)), \quad 1\leq k\leq m/d\]
is a polynomial and $C_k(t)$ 
\end{comment}
Recall that $\tilde C_k(w)$ is a balanced chain of $h$ and the right-hand part of \eqref{g0} is a polynomial.
Then, by Proposition~\ref{prop:cadena} a particular solution of \eqref{g0} is given by 
\[
\frac{(p_{i_0} \circ h) (z)}{\sum_{i \in B_{i_0}} n_i} =  \frac{\tilde g(h(z))}{d-1},
\]
and the general solution is given by adding general solutions of $\int_{\tilde C_{i_0}(w)}u=0$.

\end{proof}

\end{document}